\documentclass[%
	DIV=10, 
	english, 
	titlepage=false, 
	fontsize=11pt,
]{scrartcl}

\usepackage{todonotes}

\usepackage[utf8]{inputenc}
\usepackage[T1]{fontenc}
\usepackage{amsmath}
\usepackage{amssymb}
\usepackage{amsthm}
\usepackage{amsfonts}
\usepackage{dsfont}
\usepackage{babel}
\usepackage{lmodern}
\usepackage{csquotes}
\usepackage{enumitem}
\usepackage{array}
\usepackage{amsfonts}
\usepackage{mathtools}
\usepackage{esint}
\usepackage{array}
\usepackage{authblk}
\usepackage{nccmath}
\usepackage{bold-extra}
\usepackage{gensymb}
\usepackage{cite}
\usepackage{overpic}
\usepackage{placeins}

\usepackage{color}
\definecolor{LinkColor}{rgb}{0,0,1}
\definecolor{LinkColor2}{rgb}{0,0.5,0}
\definecolor{lg}{rgb}{.5,.5,.5}
\definecolor{rosso}{rgb}{0.8,0,0}

\newcommand{\tl}{\textcolor{blue}}

\usepackage[%
left=1.25in,
right=1.25in,
bottom=1.2in,
top=1in,
footskip=0.5in,
]{geometry}

\usepackage[%
pdftitle={Titel},%
pdfauthor={Autor},%
pdfcreator={LaTeX, LaTeX with hyperref and KOMA-Script},
pdfsubject={Betreff}, 
pdfkeywords={Keywords}
]{hyperref} 
\hypersetup{%
	colorlinks	=true,
	linkcolor	=LinkColor,%
	citecolor	=LinkColor2,%
	urlcolor	=LinkColor,%
}

\setkomafont{section}{\centering\Large\rmfamily\normalfont\scshape}
\setkomafont{subsection}{\large\rmfamily\normalfont\bfseries}
\setkomafont{subsubsection}{\rmfamily\normalfont\bfseries}
\rmfamily
\numberwithin{equation}{section}

\makeatletter
\renewcommand{\@seccntformat}[1]{\csname the#1\endcsname.\hspace{1ex}}
\makeatother

\newtheorem*{Abs}{Abstract}
\newtheorem{Thm}{Theorem}[section]
\newtheorem{Lem}[Thm]{Lemma}
\newtheorem{Pro}[Thm]{Proposition}
\newtheorem{Cor}[Thm]{Corollary}
\newtheorem{Def}[Thm]{Definition}

\theoremstyle{definition}
\newtheorem{Rem}[Thm]{Remark}

\makeatletter
\renewenvironment{proof}[1][\proofname]{%
	\par\pushQED{\qed}\normalfont%
	\topsep6\p@\@plus6\p@\relax
	\trivlist\item[\hskip\labelsep\bfseries#1\@addpunct{.}]%
	\ignorespaces
}{%
	\popQED\endtrivlist\@endpefalse
}
\makeatother

\makeatletter
\renewcommand\paragraph{\@startsection{paragraph}{4}{\z@}%
	{1ex \@plus1ex \@minus.2ex}%
	{-1em}%
	{\normalfont\normalsize\bfseries}}
\renewcommand\subparagraph{\@startsection{paragraph}{4}{\z@}%
	{1ex \@plus1ex \@minus.2ex}%
	{-1em}%
	{\normalfont\normalsize\itshape}}
\makeatother

\newcommand{\R}{\mathbb{R}}
\newcommand{\N}{\mathbb{N}}
\newcommand{\LL}{\mathcal{L}}
\newcommand{\HH}{\mathcal{H}}
\newcommand{\EE}{E}

\newcommand{\eps}{\varepsilon}

\newcommand{\dt}{\;\mathrm dt}
\newcommand{\dLL}{\;\mathrm d\mathcal{L}}
\newcommand{\dHH}{\;\mathrm d\mathcal{H}}

\newcommand{\intO}{\int_\Omega}

\newcommand{\del}{\partial}
\newcommand{\Grad}{\nabla}
\newcommand{\Lap}{\Delta}

\newcommand{\mres}{\mathbin{\vrule height 1.6ex depth 0pt width
0.13ex\vrule height 0.13ex depth 0pt width 1.3ex}}

\newcommand{\ov}{\overline}

\newcommand{\suchthat}{\;\ifnum\currentgrouptype=16 \middle\fi|\;}

\newcommand{\abs}[1]{\ensuremath\left|#1 \right|}
\newcommand{\bigabs}[1]{\ensuremath\big|#1 \big|}
\newcommand{\norm}[1]{\ensuremath\left\|#1 \right\|}
\newcommand{\bignorm}[1]{\ensuremath\big\|#1 \big\|}

\newcommand{\blk}{\color{black}}

\newcommand{\limse}{\underset{\varepsilon\searrow 0}{\lim\,}}
\newcommand{\limsupse}{\underset{\varepsilon\searrow 0}{\limsup\,}}
\newcommand{\liminfse}{\underset{\varepsilon\searrow 0}{\liminf\,}}

\newcommand{\ie}[1]{#1_{\varepsilon}}

\newcommand{\B}[1]{\boldsymbol{#1}}

\begin{document} 

\begin{titlepage}
	\begin{addmargin}{0.5 cm}
		\centering
		\LARGE{\scshape A phase-field version \\ of the Faber--Krahn theorem}\\
		\rmfamily\mdseries
		\vspace{0.02\paperheight} 
		\normalsize
		\textsc{Paul H\"uttl%
			\footnote{Fakul\"at f\"ur Mathematik, Universit\"at Regensburg, 93053 Regensburg, Germany. 
        	\href{mailto:paul.huettl@ur.de}{paul.huettl@ur.de},
        	\href{mailto:patrik.knopf@ur.de}{patrik.knopf@ur.de}}, 
    	Patrik Knopf%
    		\footnotemark[1] 
    	and
    	Tim Laux%
    		\footnote{Hausdorff Center for Mathematics, University of Bonn, 53115 Bonn, Germany. \href{mailto:tim.laux@hcm.uni-bonn.de}{tim.laux@hcm.uni-bonn.de}}}\\[1ex]
		\smallskip 
		\begin{center}
			\small
			\color{white}
			{
				\textit{This is a preprint version of the paper. Please cite as:} \\  
				P.~H\"uttl, P.~Knopf, and T.~Laux [Journal] (2021) \\ 
				\texttt{https://doi.org/...}
			}
		\end{center}
		\smallskip
		\small
		\begin{Abs}
		\normalfont
	    We investigate a phase-field version of the Faber--Krahn theorem based on a phase-field optimization problem introduced in Garcke et al.~\cite{GHKK} formulated for the principal eigenvalue of the Dirichlet--Laplacian. 
	    The shape, that is to be optimized, is represented by a phase-field function mapping into the interval $[0,1]$. 
	    We show that any minimizer of our problem is a radially symmetric-decreasing phase-field attaining values close to $0$ and $1$ except for a thin transition layer whose thickness is of order $\eps>0$. 
		Our proof relies on radially symmetric-decreasing rearrangements and corresponding functional inequalities.
		Moreover, we provide a $\Gamma$-convergence result which allows us to recover a 
        variant of the Faber--Krahn theorem for sets of finite perimeter in the sharp interface limit.
		\\[1ex]
		\end{Abs}
		\\[1ex]
		\\
	\end{addmargin}
\end{titlepage}


\bigskip
\normalsize
\setlength{\parskip}{1ex}
\setlength{\parindent}{0ex}
\allowdisplaybreaks


\normalsize

\section{Introduction}\label{SEC:INTRO}

The \textit{Faber--Krahn theorem} states that the principal eigenvalue $\lambda(E)>0$ of the eigenvalue problem
\begin{subequations}
\label{EVP:RFK}
\begin{alignat}{2}
    \label{EVP:RFK:1}
    - \Lap w &= \lambda w 
    &&\quad \text{in $E$},
    \\
    \label{EVP:RFK:2}
    w &= 0
    &&\quad \text{on $\del E$}
\end{alignat}
\end{subequations}
among all open sets $E\subset \R^n$ with $|E|=1$,
becomes minimal if $E$ is a ball.
In dimension $n=2$, this result was first conjectured by Lord Rayleigh \cite{Rayleigh} and proved independently (under suitable regularity assumptions on the boundary of $E$) by Faber~\cite{Faber} and Krahn~\cite{Krahn}.
The result holds true in every dimension $n\ge 2$ for general open sets $E$ (see, e.g., \cite[Section~3.2]{Henrot}). 
In other words, if $B$ is an open ball in $\R^n$, then the estimate
\begin{align}\label{eq:FK}
    |B|^{\frac2n} \lambda(B) \leq |E|^{\frac2n}\lambda(E)
\end{align}
holds for every open set $E\subset \R^n$ (cf.~\cite{FuscoFK}).
The latter result is referred to as the \textit{Faber--Krahn inequality}.

In this paper, we prove a diffuse interface version of this celebrated result where the boundaries of such sets $E$ are approximated by a thin interfacial layer.
Phase-fields are natural candidates to relax such minimization problems for the purpose of numerical implementation in shape and topology optimization.
The diffuse interface relaxation for eigenvalue optimization problems via phase-fields was introduced by Bogosel and Oudet \cite{BogoselOudet} and Garcke et al.~\cite{GHKK}.

In the following, we consider a fixed design domain $\Omega= B_R(0)$ that is an open ball in $\R^n$ centered at the origin with a given finite radius $R>0$.
Instead of directly working with open sets $E\subset \Omega$, we replace any given set $E$ by a function $\varphi_\eps\colon \Omega \to [0,1]$, the so-called phase-field, such that the region where $\varphi_\eps\approx 1$ approximates $E$ and the region where $\varphi_\eps\approx 0$ approximates $\Omega\backslash E$. These regions are separated by a thin transition layer whose width is related to the (usually small) interface parameter $\eps>0$. For more details, we refer to Section \ref{SEC:MAIN}.

For any phase-field function $\varphi:\Omega\to[0,1]$, we consider the eigenvalue problem
\begin{subequations}
\label{EVP:INTRO}
\begin{alignat}{2}
    \label{EVP:INTRO:1}
    - \Lap w + b^\eps(\varphi) w &= \lambda w 
    &&\quad \text{in $\Omega$},
    \\
    \label{EVP:INTRO:2}
    w&= 0
    &&\quad \text{on $\del\Omega$},
\end{alignat}
\end{subequations}
where $b^\eps$ is a suitable coefficient function driving $w$ to zero in the set $\{\varphi =0\}$ (see Section~\ref{SUBSEC:ASS} and \cite{GHKK}). 
To optimize the principal eigenvalue $\lambda_1^{\eps,\varphi}$ of this eigenvalue problem, we want to minimize the cost functional
\begin{align}
    J^\eps_\gamma(\varphi) 
        = \lambda_1^{\eps,\varphi} + \gamma \EE^\eps(\varphi),
\end{align}
(inspired by \cite{BogoselOudet} and \cite{GHKK}) over a class of admissible functions $\varphi$. 
The expression $\gamma \EE^\eps(\varphi)$ is added as a regularization term in order to make the optimization problem well-posed. 
Here, the surface tension $\gamma$ is a positive constant and $\EE^\eps(\varphi)$ stands for the Ginzburg--Landau energy associated with $\varphi$ (see \eqref{DEF:EGL}). 
We point out that the Ginzburg--Landau energy can be regarded as a diffuse interface approximation of the  perimeter functional (cf. \cite{Modica-Mortola,Modica}). The idea of perimeter regularization in shape optimization was first introduced in \cite{Ambrosio}.

Our main results show that all minimizers $\varphi_\eps$ of this optimization problem are radially symmetric-decreasing functions which indeed exhibit a phase-field structure as described above (see Theorems~\ref{THM:MIN} and \ref{THM:DifInt}). 
This radial symmetry of the phase-fields is the natural analogue to the radial symmetry of the balls in the Faber--Krahn inequality. 
Combining our symmetry result with the $\Gamma$-convergence of Garcke et al.~\cite{GHKK}, which we extend to the case of homogeneous Dirichlet boundary data for the phase-field variable by exploiting arguments of Bourdin and Chambolle \cite{BourdinChambolle}, we recover a 
variant of the Faber--Krahn theorem in the framework of  
sets of finite perimeter
(see Theorem~\ref{THM:EpsCon} and Corollary~\ref{COR:FK}) in the sharp interface limit. 
However, our result states the stronger fact that the phase-field approximation exactly captures the symmetry properties of the sharp interface limit at any fixed scale $\eps>0$, see also Theorem~\ref{THM:FKDif}. To obtain the $\Gamma$-convergence in Theorem~\ref{GammaOwen} we shortly revisit the proof in \cite{Sternberg} and adapt it to the case of general potentials, allowing for a simultaneous treatment of smooth double-well and non-smooth double-obstacle potentials.
As we consider homogeneous Dirichlet boundary data, we can construct a recovery sequence in the spirit of Modica~\cite{Modica} and Sternberg~\cite{Sternberg} by using the profile resulting from the aforementioned potential and cut it off as in \cite{BourdinChambolle}.

The proof of the main result presented in Theorem~\ref{THM:MIN} relies on a symmetrization technique where the phase-fields and the corresponding eigenfunctions of \eqref{EVP:INTRO} are compared with their radially symmetric-decreasing rearrangements (which are also referred to as \textit{Schwarz rearrangements} in the literature).
These rearrangements are at the heart of the proofs~\cite{Faber,Krahn} and have led to breakthrough results such as the quantitative isoperimetric inequality~\cite{FuscoII}.
In our diffuse interface setting, we show that the principal eigenvalue $\lambda_1^{\eps,\varphi}$
and the Ginzburg--Landau energy $\EE^\eps(\varphi)$, which constitute the cost functional, are non-increasing under radially symmetric-decreasing rearrangements.
This can be used to establish our phase-field version of the Faber--Krahn theorem. 
As a byproduct, we also obtain a phase-field version of the \textit{Euclidean isoperimetric problem}, which states that among all measurable sets of fixed volume, the ball has minimal perimeter (cf.~\cite{Maggi}).

Variants of the Faber--Krahn theorem also for other types of boundary conditions have been studied from a large variety of different viewpoints. An interesting approach from the perspective of free boundary problems is given in the alternative proof of the classical Faber--Krahn inequality by Bucur and Freitas \cite{BucurFreitas}. Therein, methods developed by Alt and Caffarelli \cite{AltCafarelli} (see also \cite{BucurVelichkov} for a comprehensive overview) are exploited to analyze the free boundary $\partial\left\{ w>0\right\}$, where $w$ is an eigenfunction to the principal eigenvalue of the Dirichlet Laplacian. Furthermore, the authors do not rely on symmetric rearrangements but rather on reflection arguments to prove the radial symmetry of optimal shapes. 
For Robin-type boundary conditions, Daners \cite{Daners} established the Faber--Krahn inequality via a level-set characterization of the cost functional.
This result was further generalized by Bucur and Daners \cite{BucurDaners} for the $p$-Laplacian subject to a Robin boundary condition. In order to avoid Lipschitz regularity in the class of admissible sets Bucur and Giacomini \cite{BucurGiacomini} interpret the Faber--Krahn inequality for the Robin Laplacian as a free discontinuity problem in the space $SBV$.\\
The study of shape optimization problems for Neumann eigenvalues on the other hand dates back to the pioneering works by Szeg\H{o} \cite{Szego} and Weinberger \cite{Weinberger}, who proved the analogon of inequality \eqref{eq:FK} for the maximization of the first non-trivial Neumann eigenvalue, which thus is often referred to as the Szeg\H{o}--Weinberger inequality, see also \cite{Henrot}. Working solely with Neumann boundary conditions induces severe instabilities for general domain perturbations setting it apart from the more classical Dirichlet- and Robin-type shape optimization problems, see \cite{Bucur}. Nevertheless, the maximization of Neumann eigenvalues shows very recent progress. Bucur and Henrot \cite{BucurHenrot} proved the natural extension of the Szeg\H{o}--Weinberger inequality for the second non-trivial Neumann eigenvalue, where now the maximum is precisely attained by the union of two disjoint equal balls. The method used there, in order to overcome the non-applicability of classical $\gamma$-convergence and lack of compactness, consists in using a relaxed notion of Neumann eigenvalues in the framework of so called \emph{degenerate densities}. In this framework Bucur et al.\  \cite{BucurMartinetOudet} proved existence results for relaxed Neumann eigenvalues. We believe that our phase-field approach linked with the ideas used there is a promising way to also tackle optimization of Neumann eigenvalues in the context of sharp interface $\Gamma$-limits in the future.

An interesting open question is a quantitative version of our result, that is, the stability of the phase-field version of the Faber--Krahn inequality. 
In the sharp interface case, this question has been settled by Brasco, De~Philippis, and Velichkov~\cite{DePhilippis} who show that the deficit in \eqref{eq:FK} can be bounded from below by the squared Fraenkel asymmetry. 
This optimal result is achieved by generalizing the selection principle of Cicalese and Leonardi~\cite{Cicalese} from the context of isoperimetric problems to the case of eigenvalues. 
It should be expected that the symmetrization procedure applied in our work would result in a (suboptimal) quantitative version of our inequality, just as in the case of the sharp interface by Fusco, Maggi, and Pratelli~\cite{FuscoFK}, who rely on symmetrization and their quantitative isoperimetric inequality~\cite{FuscoII}.

Finally, let us also mention another interesting variant of the Faber--Krahn inequality, namely the Pólya--Szeg\H{o} conjecture (see~\cite[p.~159]{PS}), which states that among all planar polygons of fixed enclosed area, the regular polygon minimizes the first Dirichlet--Laplace eigenvalue. 
Only recently, some significant progress has been achieved by Bogosel and Bucur~\cite{bogosel2022polygonal}, which indicates that also here, the most symmetric configuration yields the smallest eigenvalue. This approach might lead to a computer-assisted proof of the conjecture, at least for $n$-gons with moderately small $n$.

Our paper is structured as follows. 
First we fix our notation and gather some preliminary results in Section~\ref{SEC:PRE}. 
In Section~\ref{SEC:MAIN}, we precisely formulate the problem and state our main results. 
Finally, in Section~\ref{SECT:PROOF}, we provide the proofs of our results.

\section{Preliminaries and important tools}\label{SEC:PRE}

\subsection{Notation}\label{SUBSEC:NOT}
We write $\R_0^+=[0,\infty)$ to denote the interval of non-negative real numbers.
The interval $[0,+\infty]$ is to be understood as a subset of the extended real numbers $\overline\R = \R\cup\{\pm\infty\}$, on which we use the standard convention $\pm \infty \,\cdot\, 0 = 0$. For any $n\in\N$, $\LL^n$ stands for the $n$-dimensional Lebesgue measure and $\HH^n$ denotes the $n$-dimensional Hausdorff measure. 

\subsection{Assumptions}\label{SUBSEC:ASS}
Note that in the upcoming analysis we will always choose the design domain $\Omega=B_R(0)$ that is the open ball in $\R^n$ with radius $R>0$ centered at the origin. Furthermore, the following assumptions on the potential $\psi:\R\to\R\cup\left\{+\infty\right\}$ and the coefficients $b^\eps$ are supposed to hold throughout this paper:
\begin{enumerate}[label = (A\arabic*), leftmargin=*]
    \item \label{ASS:A:1}
    $\psi\in C^2([0,1])$,
    $\psi(0)=\psi(1)=0$ and $\psi>0$ in $(0,1)$.
    \item \label{ASS:A:1b}
    The minima of $\psi$ at $0$ and $1$ are non-degenerate in the sense that for $x\in\left\{0,1\right\}$, we either have $\psi'(x)\neq 0$ or $\psi'(x)=0<\psi^{\prime\prime}(x)$.
    \item \label{ASS:A:2}
    For any $\eps>0$, the coefficient $b^\eps$ is a function
    \begin{align}
        b^\eps:[0,1] \to [0,\beta^\eps]
    \end{align}
    for some real number $\beta^\eps>0$. We demand that $b^\eps$ is continuous, strictly decreasing and surjective onto $[0,\beta^\eps]$.

    \item \label{ASS:A:3}
    The numbers $\beta^\eps=b^\eps(0)$ satisfy
    \begin{align}
        \underset{\eps\searrow 0}{\lim}\; \beta^\eps 
        = + \infty
        \quad\text{and}\quad
        \beta^\eps = o\big(\eps^{-\kappa}\big) 
        \quad\text{with}\quad
        \begin{cases}
            \kappa\in (0,1)
            &\text{if}\; n = 2, \\
            \kappa = \frac 2n 
            &\text{if}\; n\ge 3.
        \end{cases}
    \end{align}
    Moreover, there exists a limit function     
    \begin{align}
        b^0:[0,1] \to [0,+\infty]
    \end{align}
    satisfying
    \begin{itemize}
        \item $b^0\big(\tfrac 12\big)<+\infty$,
        \item $b^{\varepsilon}\to b^0$ pointwise in $[0,1]$ as $\eps\to 0$,
        \item $b^{\delta} \ge b^{\varepsilon}$ on $[0,1]$ for all $0\le\delta\le\varepsilon$.
    \end{itemize}
\end{enumerate}

\begin{figure}[ht!]
    \centering

        \begin{tikzpicture}[scale=4]
            \draw[thick,->] (0,-.02) node[below]{$0$} -- (0,1);
            \draw[thick,->] (-0.02,0) -- (1.2,0) node[below] {$\varphi$};

            \draw[thick] (-.02,{1/2/2.7/(0.6)/(0.9)}) node[left] {$\beta^{\varepsilon}$} -- (.02,{1/2/2.7/(0.6)/(0.9)});
            
            \draw[thick] (-.02,{1/2/2/(0.5)/(0.7)}) node[left] {$\beta^{\delta}$} -- (.02,{1/2/2/(0.5)/(0.7)});
            
            \draw[thick,domain=1:0.14,smooth,variable=\x,red] plot ({\x},{1/2/2/(\x+0.3)/(\x+0.4)*(1-\x*\x))}) node[right] {$b^0$};

            \draw[thick,domain=1:0,smooth,variable=\x,magenta] plot ({\x},{1/2/2/(\x+0.5)/(\x+0.7)*(1-\x*\x))}) node[right] {$b^{\delta}$};
            
            \draw[thick,domain=1:0,smooth,variable=\x,blue] plot ({\x},{1/2/2.7/(\x+0.6)/(\x+0.9)*(1-\x*\x))});
            
            \node[blue] at (.1,.17) {$b^\varepsilon$};
            
            \draw[thick] (1,-.02) node[below] {$1$} -- (1,.02);

        \end{tikzpicture}
    $\quad$
        \begin{tikzpicture}[scale=4]
            \draw[thick,->] (0,-.02) node [below] {$0$} -- (0,1);
            \draw[thick,->] (-0.02,0) -- (1.2,0) node[below] {$\varphi$};
            \draw[thick,domain=0:1,smooth,variable=\x,red] plot ({\x},{10*\x*\x*(1-\x)*(1-\x)});
            \draw[thick,domain=0:1,smooth,variable=\x,blue] plot ({\x},{4*\x*(1-\x)});    
            \node[red] at (0.5,10*0.5*0.5*0.5*0.5) [above] {$\psi(\varphi)$};
            \node[blue] at (0.5,4*0.5*0.5) [above] {$\psi(\varphi)$};
            \draw[thick] (1,-.02) node[below] {$1$} -- (1,.02);
        \end{tikzpicture}
    \caption{Sketch of a possible choice for the coefficient functions $b^\eps$ and $b^0$ (left) and the potential $\psi$ (right).}
    \label{fig:b_eps}
\end{figure}
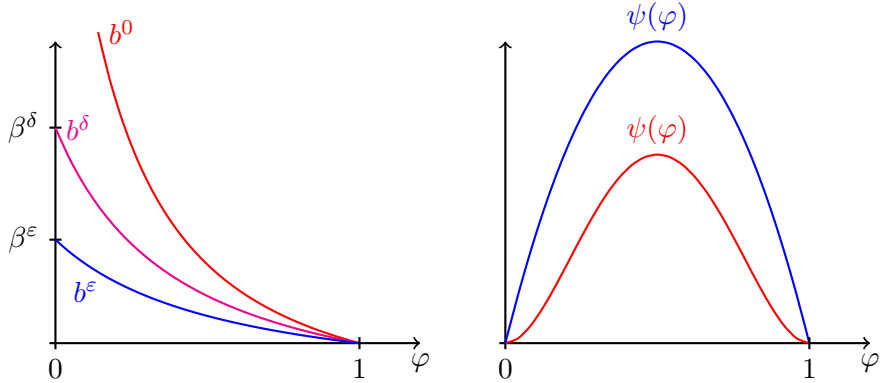

\begin{Rem}\label{Rem:Pot}
The two classical choices we have in mind for the potential $\psi$ are either the smooth quartic double-well potential $\psi(x)=\frac{1}{4}(1-x)^2x^2$
(which satisfies $\psi'(x)=0<\psi''(x)$ for $x\in\{0,1\}$)
or the non-smooth double-obstacle potential
\begin{align*}
    \psi(x)=\begin{cases}
    \frac{1}{2}(1-x)x&\quad \text{if } x\in [0,1],\\
    +\infty&\quad \text{else}.
    \end{cases}
\end{align*}
(which satisfies $\psi'(x)\neq 0$ for $x\in\{0,1\}$).
However, the assumptions \ref{ASS:A:1} and \ref{ASS:A:1b} allow for very general potentials.
In particular, asymmetric potentials satisfying $\psi^{\prime}(0)\neq 0$ and $\psi^{\prime}(1)=0<\psi''(1)$ (or vice versa) can also be included.

Note that in the case of a smooth potential as studied in \cite{Modica,Sternberg}, in contrast their theory, we do not need a growth condition as in \cite[Proposition 3(b)]{Modica} or \cite[Proposition 3]{Sternberg} since we additionally incorporate the box constraint $\varphi\in [0,1]$ in our set of admissible phase-fields. Therefore, depending on the choice of $\psi$, one of the results \cite[Proposition 3(a)]{Modica}, \cite[Remark (1.35)]{Sternberg} and \cite[Theorem 3.7]{BloweyElliott} can be applied and yields compactness of the Ginzburg--Landau energy.

The assumptions on $\psi$ in \cite[Theorem 1]{Sternberg} differ from \ref{ASS:A:1} only in the fact that global continuity is assumed. However, due to the box constraint $\varphi\in [0,1]$, our phase-fields may not leave the interval $[0,1]$ and thus, such an assumption is not necessary. 

The crucial difference between \cite{Sternberg} and \cite{BloweyElliott} is that in \cite{Sternberg}, the potentials need to satisfy \ref{ASS:A:1b} with $\psi'(x)=0<\psi''(x)$ for $x\in\{0,1\}$, whereas \cite{BloweyElliott} only covers the case $\psi'(x)\neq 0$ for $x\in\{0,1\}$. 
However, we will see that also the mixed case $\psi^{\prime}(0)\neq 0$ and $\psi^{\prime}(1)=0<\psi''(1)$ (or vice versa) can be handled by combining the proofs of \cite[Theorem 1]{Sternberg} and \cite[Proposition 3.11]{BloweyElliott}. 
This is possible since their construction of a recovery sequence remains practicable as the ODE
\begin{align}
    \label{ODE*}
    \eta^{\prime}(t)=\sqrt{2\psi(\eta(t))},
\end{align}
which is used to define the profile at the diffuse interface, possesses a global solution that is strictly increasing as long as $\eta(t)\in (0,1)$.
In \cite{Sternberg}, any solution of \eqref{ODE*}
satisfies $\eta(t)\in (0,1)$ for all $t\in\R$, whereas in \cite{BloweyElliott}, there exist $t_0,t_1\in\R$ with $t_0<t_1$ such that 
\begin{align}
    \label{PRF}
    \eta(t)
    \begin{cases}
        = 0, &\text{if $t\in(-\infty,t_0]$},\\
        \in (0,1) , &\text{if $t\in(t_0,t_1)$},\\
        = 1, &\text{if $t\in[t_1,\infty)$}.
    \end{cases}
\end{align}
In our $\Gamma$-convergence proof, we will take care of both cases simultaneously. 
Therefore, proceeding as in \cite{Sternberg}, we interpolate the solution $\eta$ of \eqref{ODE*} in such a way that the interpolated solution exhibits the behavior described in \eqref{PRF}. The solvability of \eqref{ODE*} and further properties of solutions to this ODE will be analyzed in depth in the proof of Theorem~\ref{GammaOwen}.

\end{Rem}

\subsection{Symmetric-decreasing rearrangements}\label{SUBSEC:SDR}

For functions $f:\R^n\to \R_0^+$ vanishing at infinity (i.e., the level sets ${\{x\in\R^n \,\vert\, f(x)>t\}}$ have finite Lebesgue-measure for all $t>0$), a definition of their radially symmetric-decreasing rearrangement can be found in \cite[Section~3.3]{lieb-loss}. We can easily adapt this definition to functions $f:\Omega \to \R_0^+$ where $\Omega= B_R(0)$ is an open ball in $\R^n$ with radius $R>0$ centered at the origin.

\begin{Def}\label{DEF:SDR}
Let $\Omega = B_R(0)$ be an open ball in $\R^n$ centered at the origin with a given radius $R>0$.
\begin{enumerate}[label = \textnormal{(\alph*)}, leftmargin=*]
    \item\label{DEF:SDR:SD}A measurable function ${f:\Omega\to\R}$ is called \textbf{(radially) symmetric-decreasing}, if any fixed representative of the equivalence class of $f$ satisfies the properties 
    \begin{align}
    \label{SYMMDEC}
    \left\{
    \begin{aligned}
        f(x) &= f(y) &&\quad\text{if}\quad \abs{x} = \abs{y},\\
        f(x) &\ge f(y) &&\quad\text{if}\quad \abs{x} \le \abs{y}
    \end{aligned}
    \right.
    \end{align}
    for almost all $x,y\in\Omega$. If additionally 
    \begin{align*}
        f(x) > f(y) \quad\text{if}\quad \abs{x} < \abs{y}
    \end{align*}
    for almost all $x,y\in\Omega$, then $f$ is called \textbf{strictly (radially) symmetric-decreasing}.%
    \item \label{DEF:SDR:A}
    For any measurable set $A\subseteq \Omega$ with $\LL^n(A)<\infty$, its \textbf{(radially) symmetric rearrangement} $A^*$ is defined to be the open ball centered at the origin whose volume is equal to that of $A$. This means that
    \begin{align*}
        A^* = \big\{ x \in \Omega\,:\, \abs{x} < r \big\}
        \quad\text{where $r\ge 0$ satisfies}\quad
        \LL^{n}(B^n)\,r^n
        = \LL^n(A).
    \end{align*}
    Here, $B^{n}$ denotes the $n$-dimensional unit ball.
    \item\label{DEF:SDR:B}  
    Let $f:\Omega \to \R_0^+$ be any measurable function.
    Then its \textbf{(radially) symmetric-decreasing rearrangement} $f^*$ is defined as
    \begin{align*}
        f^*(x) = \int_0^\infty \mathds{1}_{{\{ f>t \}}^*}(x) \dt
    \end{align*}
    for \textit{all} $x\in\Omega$.
\end{enumerate}
\end{Def}

\begin{Rem}\label{REM:SDR}
Let $\Omega = B_R(0)$ be an open ball in $\R^n$ centered at the origin with a given radius $R>0$.
\begin{enumerate}[label = \textnormal{(\alph*)}, leftmargin=*]
    \item\label{REM:SDR:A} It obviously holds that $\Omega^*=\Omega$, and for any measurable function $f:\Omega \to \R_0^+$, we have $f^{**} = f^*$.
    In particular, this motivates the particular choice $\Omega=B_R(0)$. 
    \blk
    \item\label{REM:SDR:B} For any measurable function $f:\Omega \to \R_0^+$, its trivial extension $f_0:\R^n \to \R_0^+$ with $f_0\vert_\Omega = f$ and $f_0\vert_{\R^n\backslash\Omega} = 0$ is measurable and naturally vanishes at infinity. In particular, we have $f_0^*\vert_\Omega = f^*$, where the symmetric-decreasing rearrangement $f_0^*$ of the extension $f_0$ is defined as in \cite[Section~3.3]{lieb-loss}.
\end{enumerate}
\end{Rem}

\bigskip

Some important properties of the symmetric-decreasing rearrangement are collected in the following lemma.

\begin{Lem}\label{LEM:SDR}
    Let $\Omega = B_R(0)$ be an open ball in $\R^n$ centered at the origin with a given radius $R>0$,
    and let $f,g:\Omega\to\R_0^+$ be arbitrary measurable functions.
    Then the following statements hold:
    \begin{enumerate}[label = \textnormal{(\alph*)}, leftmargin=*]
        \item \label{item:f* measurable} 
        $f^*$ is measurable and symmetric-decreasing. Moreover, $f^*$ is defined everywhere in $\Omega$. In particular, the condition \eqref{SYMMDEC} is satisfied \textit{everywhere} in $\Omega$.
        \item \label{item:f* level sets and p norms}
        The level sets of $f^*$ are the rearrangements of the level sets of $f$, meaning that
        \begin{align*}
            \big\{x\in\Omega : f^*(x) > t \big\}
            = \big\{x\in\Omega : f(x) > t \big\}^*
        \end{align*}
        up to a Lebesgue null set in $\R^n$. In particular, if $f\in L^p(\Omega)$ for some $p\in [1,\infty]$, it holds that $f^*\in L^p(\Omega)$ with
        \begin{align*}
            \norm{f^*}_{L^p(\Omega)} = \norm{f}_{L^p(\Omega)}.
        \end{align*}
    
        \item \label{item:f* phi}
        Let $\Phi:\R_0^+\to\R_0^+$ be a non-decreasing, lower semicontinuous function. Then it holds that
        \begin{align*}
            (\Phi\circ f)^* = \Phi\circ f^*
            \quad\text{a.e.~in $\Omega$}.
        \end{align*}
        
        \item \label{item:f* psi} 
        Let $\Psi\in C^1([0,1])$ with $\Psi(0)=0$. 
        If $0\le f\le 1$ almost everywhere in $\Omega$, it holds that
        \begin{align}
        \label{EQ:PSI}
            \intO \Psi\circ f^* \dLL^n 
            = \intO \Psi\circ f \dLL^n.
        \end{align} 
        
        \item \label{item:f* hardy littlewood}
        \textbf{Hardy--Littlewood inequality:} It holds that
        \begin{align}
        \label{HLE}
            \intO f\, g \dLL^n \le \intO f^*\, g^* \dLL^n
        \end{align}
        with the convention that when the left-hand side is infinite, then also the right-hand side is infinite.
        \item \label{item:f* nonexpansive}
        \textbf{Nonexpansivity of the rearrangement:}
        Let $F:\mathbb{R}\to\mathbb{R}_0^+$ be a convex function such that $F(0)=0$. Then 
        \begin{align*}
            \int_{\Omega}F\circ(f^*-g^*)\dLL^n
            \le \int_{\Omega}F\circ (f-g)\dLL^n.
        \end{align*}
        \item \label{item:f* polya szego}
        \textbf{Pólya--Szeg\H{o} inequality:}
        Suppose that $f\in H^1_0(\Omega;\R_0^+)$. Then, $f^*\in H^1_0(\Omega;\R_0^+)$ with
        \begin{align}
        \label{PSI}
            \intO \abs{\Grad f^*}^2 \dLL^n
            \le \intO \abs{\Grad f}^2 \dLL^n.
        \end{align}
        Moreover, if $f>0$ almost everywhere in $\Omega$ and
        \begin{align} 
        \label{COND:GRAD}
            \LL^n\big(\{x\in\Omega \,\vert\, \Grad f^*(x) = 0\}\big)=0,
        \end{align}
        then equality in \eqref{PSI} holds if and only if $f=f^*$ almost everywhere in $\Omega$.
    \end{enumerate}
\end{Lem}

The proof is deferred to Section \ref{SECT:PROOF}.

\begin{Rem}\label{REM:PS}
    We point out that the condition that $f$ has a vanishing trace on $\del\Omega$ is actually a necessary assumption for the Pólya--Szeg\H{o} inequality (Lemma~\ref{LEM:SDR}\ref{item:f* polya szego}). In general, as the following example shows, there exist functions $f\in H^1(\Omega;\R_0^+)$ such that $f^*$ does not even belong to $H^1(\Omega;\R_0^+)$.
    
    \textbf{Counterexample to the Pólya--Szeg\H{o} inequality for functions in $H^1(\Omega;\R_0^+)$.}\\
    Let $\Omega = B_1(0)$ be the open unit ball in $\R^2$. We consider the function
    \begin{align*}
        f:\Omega\to\R_0^+,\quad x \mapsto \abs{x}\,,
    \end{align*}
    which obviously belongs to $H^1(\Omega;\R_0^+)$ but not to $H^1_0(\Omega;\R_0^+)$. 
    Its symmetric-decreasing rearrangement $f^*$ is given by
    \begin{align*}
        f^*:\Omega\to\R_0^+,\quad x \mapsto \sqrt{1-\abs{x}^2}\,.
    \end{align*}
    Hence, $f^*$ is weakly differentiable with 
    \begin{align*}
        \Grad f^*(x) = \frac{-x}{\sqrt{1-\abs{x}^2}}
        \quad\text{for all $x\in\Omega\backslash\{0\}$.}
    \end{align*}
    However, it is easy to see that the blow-up at $|x|=1$ causes
    \begin{align*}
        \int_\Omega \abs{\Grad f^*}^2 \dLL^2 
        = +\infty.
    \end{align*}
    This means that $f^*\notin H^1(\Omega;\R_0^+)$ and in particular, the Pólya--Szeg\H{o} inequality stated in \eqref{PSI} does not hold.
\end{Rem}

\subsection{Properties of functions of bounded variation}\label{Sub:BV}
We recall some basic facts on functions of bounded variation and sets of finite perimeter that will be used in the course of this paper.  We refer to~\cite{EvansGariepy, Maggi, AFP} for more details.

The space of \textbf{functions of bounded variation} in $\Omega$ with values in $\R$, also referred to as \textbf{$BV$ functions}, is defined as
\begin{align*}
        BV(\Omega)
    \coloneqq 
        \left\{u\in L^1(\Omega)\suchthat V(u,\Omega)<\infty\right\}.
\end{align*}
Here $V(u,\Omega)$ denotes the \textbf{variation} of a function $u\in L^1_{\mathrm{loc}}(\Omega)$ defined as
\begin{align*}
    V(u,\Omega)
        \coloneqq
    \sup\left\{\int_{\Omega}u\, \operatorname{div}\B{\xi}\dLL^n \suchthat \B{\xi}\in C_0^1(\Omega,\R^n),\, \norm{\B{\xi}}_{L^{\infty}(\Omega)}\le 1\right\}.
\end{align*}
Endowed with the norm
\begin{align*}
    \norm{u}_{BV(\Omega)}\coloneqq \norm{u}_{L^1(\Omega)}+V(u,\Omega),
\end{align*}
the space $BV(\Omega)$ is a Banach space. However, for practical purposes, the topology induced by this norm is too strong. For this reason, the concept of \textbf{strict convergence} is commonly used. We say that a sequence $u_k\in BV(\Omega)$ strictly converges to $u\in BV(\Omega)$ if
\begin{align*}
    u_k\to u \quad\text{\,in } L^1(\Omega)
    \quad \text{and} \quad
    V(u_k,\Omega)\to V(u,\Omega)
\end{align*}
as $k\to \infty$.

One of the fine properties of $BV$ functions is that they allow for a well-defined trace. Due to \cite[Theorem 3.87]{AFP} any $u\in BV(\Omega)$ possesses a trace 
$$u_{\vert{\del\Omega}}\in L^1\big((\partial\Omega, \HH^{n-1}\mres \partial\Omega);\R\big),$$ 
which is defined via the limit
\begin{align}\label{DEF:TR}
    \underset{\rho\searrow 0}{\lim}\;
    \rho^{-n}\int_{\Omega\cap B_{\rho}(x)}\abs{u(y)-u_{\vert{\del\Omega}}(x)}\dLL^n(y)=0.
\end{align}
for $\HH^{n-1}$-almost every $x\in\partial\Omega$.
Here, $\HH^{n-1}\mres \partial\Omega$ denotes the restriction of the Hausdorff measure $\HH^{n-1}$ to the boundary $\del\Omega$ and $L^1\big((\partial\Omega, \HH^{n-1}\mres \partial\Omega);\R\big)$ is the space of $L^1$-functions on $\del\Omega$ with respect to the measure $\HH^{n-1}\mres \partial\Omega$. 
In the following, we will simply write $L^1(\del\Omega)$ instead of 
$L^1\big((\partial\Omega, \HH^{n-1}\mres \partial\Omega);\R\big)$.
The corresponding norm on $L^1(\del\Omega)$ is given by
\begin{align}\label{L1NormDef}
    \norm{\cdot}_{L^1(\partial\Omega)}=\int_{\partial\Omega}\abs{\cdot}\dHH^{n-1}.
\end{align}
Due to \cite[Theorem 3.88]{AFP}, the operator
\begin{align*}
    BV(\Omega)&\to L^1(\partial\Omega),\\
    u&\mapsto u_{\vert{\del\Omega}},
\end{align*}
is continuous with respect to strict convergence in $BV(\Omega)$. 
    
To conclude this section we give the definition of the relative perimeter.    
The \textbf{relative perimeter in $\Omega$} of a measurable set $E\subset \R^n$ is defined as
\begin{align*}
    P_{\Omega}(E)\coloneqq V(\chi_{E},\Omega).
\end{align*} 

\pagebreak[2]

\section{Formulation of the problem and the main results}\label{SEC:MAIN}

In the following, we consider the design domain $\Omega = B_R(0) \subset \R^n$ with $n\ge 2$
and some finite radius $R>0$. 

For functions $\varphi\in L^{\infty}(\Omega, [0,1])$,
the eigenvalue problem introduced in \cite{GHKK} reads as
\begin{subequations}
\label{EVP}
\begin{alignat}{3}
    \label{EVP:EQ}
    - \Lap w + b^\eps(\varphi) w &= \lambda w 
    &&\qquad \text{in $\Omega$}, \\
    \label{EVP:BC}
    w\vert_{\del\Omega}&= 0
    &&\qquad\text{on $\del\Omega$.}
\end{alignat}
\end{subequations}
In view of the term $b^{\eps}(\varphi)w$, this problem can be understood as a phase-field approximation of the classical Dirichlet--Laplace eigenvalue problem on the shape represented by the set $\left\{\varphi=1\right\}$. For a detailed motivation for and introduction to this eigenvalue problem we refer to \cite[Section 2]{GHKK}. To specify the notion of weak solutions, eigenvalues and eigenfunctions, we recall the following definition (cf.~\cite[Section~3.1]{GHKK}).

\begin{Def}\label{DEF:EVPROBLEM}
    Let $\eps>0$ and $\varphi\in L^{\infty}(\Omega, [0,1])$ be arbitrary. 
    \begin{enumerate}[label = \textnormal{(\alph*)}, leftmargin=*]
    \item For any given $\lambda\in\R$, a function $w\in H^1_0(\Omega)$ is called a \textbf{weak solution} of the system in \eqref{EVP} if the weak formulation
    \begin{align}\label{WeakEVe}
        \intO \Big(\Grad w \cdot \Grad \eta + b^\eps(\varphi)\, w \eta \Big)\dLL^n
        = \lambda \intO w \eta \dLL^n
    \end{align}
    is satisfied for all test functions $\eta\in H^1_0(\Omega)$.
    \item A real number $\lambda^{\eps,\varphi}$ is called an \textbf{eigenvalue} associated with $\varphi$, if there exists at least one nontrivial weak solution $w^{\eps,\varphi}\in H^1_0(\Omega)$ of the eigenvalue problem \eqref{EVP} written for $\lambda = \lambda^{\eps,\varphi}$.
    
    In this case, $w^{\eps,\varphi}$ is called an \textbf{eigenfunction} associated with the eigenvalue $\lambda^{\eps,\varphi}$. 
    \end{enumerate}
\end{Def}

We further recall some important properties of the eigenvalue problem \eqref{EVP}. The results can be found in \cite{GHKK}, but are also accessible via the standard literature (see, e.g., \cite{Alt,Gilbarg-Trudinger}).

\begin{Pro}\label{PROP:EVP}
    Let $\eps>0$ and $\varphi\in L^{\infty}(\Omega, [0,1])$ be arbitrary. 
    \begin{enumerate}[label = \textnormal{(\alph*)}, leftmargin=*]
    \item \label{PROP:EVP:A}
    The eigenvalue problem \eqref{EVP} has countably many eigenvalues and each of them has a finite-dimensional eigenspace. Repeating each eigenvalue according to its multiplicity, we can write them as a sequence $\big(\lambda_k^{\eps,\varphi}\big)_{k\in\N}$ with
    \begin{align*}
        0<\lambda_1^{\eps,\varphi}
        \le\lambda_2^{\eps,\varphi}
        \le\lambda_3^{\eps,\varphi}
        \le ...
        \quad\text{and}\quad
        \lambda_k^{\eps,\varphi} \to \infty
        \;\;\text{as}\;\;
        k\to\infty.
    \end{align*}
    
    \item \label{PROP:EVP:B} 
    There exists an orthonormal basis $\big(w_k^{\eps,\varphi}\big)_{k\in\N}$ of $L^2(\Omega)$ where for every $k\in\N$, $w_k^{\eps,\varphi}$ is an eigenfunction to the eigenvalue $\lambda_k^{\eps,\varphi}$.
    
    \item \label{PROP:EVP:C}
    The eigenvalue $\lambda_1^{\eps,\varphi}$ is called the \textbf{principal eigenvalue}. It can be represented via the Courant--Fischer characterization
    \begin{align}
        \label{CHAR:CF}
        \lambda_1^{\eps,\varphi}
        = \underset{w \in H^1_0(\Omega)\backslash\{0\}}{\min}
        \frac{\intO \big(\abs{\Grad w}^2 + b^\eps(\varphi) w^2 \big)\dLL^n}{\norm{w}_{L^2(\Omega)}^2}\;.
    \end{align}
    Any function $w \in H^1_0(\Omega)\backslash\{0\}$ at which this minimum is attained is an eigenfunction to the eigenvalue $\lambda_1^{\eps,\varphi}$.
    
    Moreover, the eigenspace of $\lambda_1^{\eps,\varphi}$ is one-dimensional, and there exists a unique eigenfunction $\ov w \in H^1_0(\Omega)\backslash\{0\}$ corresponding to this eigenvalue which fulfills
    \begin{align}
    \label{COND:POSNORM}
        \ov w >0 \quad\text{a.e.~in $\Omega$},
        \quad\text{and}\quad
        \norm{\ov w}_{L^2(\Omega)}=1.
    \end{align}
    We call $\ov w$ the \textbf{positive-normalized eigenfunction}. Without loss of generality, as the choice the sign does not matter, we will always choose $w_1^{\eps,\varphi}=\ov w$ in the orthonormal basis given by part \ref{PROP:EVP:B}.
    \end{enumerate}
\end{Pro}

For any prescribed $m\in (0,1)$, we define the set of admissible controls
\begin{align*}
    \Phi_m 
    := \left\{
        \varphi \in H^1_0(\Omega)
        \suchthat
        \begin{aligned}
        &0 \le \varphi(x) \le 1 \;\;\text{for a.e. $x\in\Omega$}, \\
        & \textstyle\fint_\Omega \varphi \dLL^n = m
        \end{aligned}
    \right\}
    \subset H^1_0(\Omega) \cap L^\infty(\Omega)
    \,.
\end{align*}

Applying \cite[Theorem 8.12]{Gilbarg-Trudinger} we directly infer the following statement.
\begin{Cor} \label{COR:REG}
    Let $\eps>0$ and $\varphi\in\Phi_m$ be arbitrary, and let $w$ be an eigenfunction associated with the eigenvalue $\lambda_1^{\eps,\varphi}$ in the sense of Definition~\ref{DEF:EVPROBLEM}.
    Then it holds that $w\in H^1_0(\Omega)\cap H^2(\Omega)$ and $w$ is a strong solution of the system in \eqref{EVP}, meaning that
    \begin{subequations}
    \begin{alignat*}{3}
        - \Lap w + b^\eps(\varphi) w &= \lambda_1^{\eps,\varphi} w 
        &&\qquad \text{a.e.~in $\Omega$}, \\
        w\vert_{\del\Omega} &= 0
        &&\qquad\text{a.e.~on $\del\Omega$.}
    \end{alignat*}
    \end{subequations}
\end{Cor}

Now we formulate the shape optimization problem for the principal eigenvalue. This can be regarded as a special case of the framework in \cite{GHKK} by choosing $\Psi(\lambda_1)=\lambda_1$ there. Hence, we only briefly summarize the main aspects concerning this optimization problem at this point.

For $\eps>0$ and $\varphi\in\Phi_m$, we now introduce the Ginzburg--Landau energy
\begin{align}
    \label{DEF:EGL}
    \EE^\eps(\varphi) := \intO \Big(\frac{\eps}{2} \abs{\Grad\varphi}^2
        + \frac{1}{\eps} \psi(\varphi) \Big)\dLL^n\,.
\end{align}
This term regularizes the optimization problem in order for it to be well-posed (see e.g., \cite[Theorem 6.1]{GHK}).\\
We observe that the Ginzburg--Landau energy is decreasing with respect to symmetric-decreasing rearrangement of its argument. This can be interpreted as a phase-field version of the isoperimetric inequality.

\begin{Lem}[Phase-field isoperimetric inequality]\label{LEM:EGL}
    Let $\eps > 0$ be arbitrary. Then, for all $\varphi \in H_0^1(\Omega;[0,1])$, we have
    \begin{align}\label{EST:E*}
        \EE^\eps(\varphi^*) \le \EE^\eps(\varphi).
    \end{align}
\end{Lem}

Furthermore we will prove the following phase-field version of the Faber--Krahn inequality on the diffuse interface level.
\begin{Thm}[Phase-field Faber--Krahn inequality]\label{THM:FKDif}
    Let $\eps > 0$ be arbitrary. Then, for all $\varphi \in H_0^1(\Omega;[0,1])$, we have
    \begin{align*}
        \lambda_1^{\eps,\varphi^*} \le \lambda_1^{\eps,\varphi}.
    \end{align*}
\end{Thm}

In order to recover the classical Faber--Krahn inequality in the sharp interface limit $\eps\to 0$, we consider the following optimization problem:
\begin{align}
    \label{OP}
    \left\{
    \begin{aligned}
        &\text{Minimize}
        && J^\eps_\gamma(\varphi) 
        = \lambda_1^{\eps,\varphi} + \gamma \EE^\eps(\varphi) \\
        &\text{subject to} 
        && \varphi\in \Phi_m\,.
    \end{aligned}
    \right.
    \tag{$\text{OP}^\eps_\gamma$}
\end{align}
Here, $\lambda_1^{\eps,\varphi}$ denotes the principal eigenvalue corresponding to the function $\varphi$ as introduced in Proposition~\ref{PROP:EVP}, and $\gamma>0$ is the surface tension. Here, the additional summand $\gamma \EE^\eps(\varphi)$ acts as a regularization term which ensures well-posedness of the optimization problem and is further used to gain additional information about its minimizers.
More precisely, for fixed $\eps>0$, the gradient term appearing in the Ginzburg--Landau energy ensures the weak compactness of any minimizing sequence of phase-fields in $H_0^1(\Omega)$, which is needed to apply the direct method in the calculus of variations (see e.g. \cite[Theorem~6.1]{GHK}).

After passing to the sharp interface limit $\eps\to 0$, we recover the classical Faber--Krahn inequality in the framework 
of sets of finite perimeter which are represented by $BV$ functions.
First of all for fixed $\gamma>0$, the Ginzburg--Landau energy gives rise to compactness in the space $BV$ for a sequence of minimizers $(\varphi_\eps)_{\eps>0}$ of $J^\eps_{\gamma}$ as $\eps\to 0$, thus providing us with a minimizer $\varphi_0\in BV(\Omega;\{0;1\})$ on the sharp interface level, see the proof of Theorem~\ref{THM:EpsCon} for details.
Afterwards we are able to send $\gamma\to 0$ in order to get rid of the additional perimeter regularization,
\blk 
which is possible as we will see later that the minimizer on the sharp interface level does not depend on $\gamma$.  

The existence of a minimizer $\varphi\in \Phi_m$ of the optimization problem \eqref{OP} was established in \cite[Theorem~3.8]{GHKK}.
This means that the following lemma holds.

\begin{Lem}\label{LEM:MIN}
    Let $\eps,\gamma>0$ be arbitrary.
    Then, the optimization problem \eqref{OP} possesses a minimizer $\varphi\in \Phi_m$.
\end{Lem}

Now, based on Lemma~\ref{LEM:EGL} and Theorem~\ref{THM:FKDif},
the next theorem shows that minimizers of \eqref{OP} are necessarily symmetric-decreasing. The same holds for the positive-normalized eigenfunction of the corresponding principal eigenvalue.

\begin{Thm}[Phase-field Faber--Krahn]\label{THM:MIN}
    Let $\eps,\gamma>0$ be arbitrary, $m\in(0,1)$. Also, let $\varphi\in \Phi_m$ be any minimizer of the optimization problem \eqref{OP}.
    Then, $\varphi=\varphi^*$ almost everywhere in $\Omega$, meaning that $\varphi$ is symmetric-decreasing, and the positive-normalized eigenfunction $w_1^{\eps,\varphi}$ to the principal eigenvalue $\lambda_1^{\eps,\varphi}$ also fulfills $w_1^{\eps,\varphi}=(w_1^{\eps,\varphi})^*$ almost everywhere in $\Omega$. 
\end{Thm}

Furthermore, the following theorem, which is a direct consequence of the boundedness of the Ginzburg--Landau energy along a sequence of minimizers for $\eps\to 0$, shows that the thickness of the interface up to an infinitesimally small error is $\mathcal{O}(\eps)$.

\begin{Thm}\label{THM:DifInt}
   Let $\gamma>0$ and $m\in (0,1)$ be arbitrary. 
   Then, there exists a constant $C>0$ such that for any minimizer $\varphi_\eps$ of the optimization problem \eqref{OP} with $\eps>0$ and all $0<\delta<\frac 12$, it holds that
   \begin{align}
        \label{EST:INTERFACE}
        \LL^n\left(\left\{\delta\le \varphi_\eps 
        \le 1-\delta\right\}\right)\le \frac{C\eps}{\alpha_\delta \gamma }
        \quad\text{with}\quad
        \alpha_\delta:= \underset{[\delta,1-\delta]}{\min} \;\psi >0.
   \end{align}
\end{Thm}

The proofs of Theorems~\ref{THM:MIN} and \ref{THM:DifInt} are presented in Section~\ref{SECT:PROOF}.

Combining the preceding two results, we deduce that every minimizer $\varphi_\eps$ of \eqref{OP} is symmetric-decreasing and exhibits the expected phase-field structure, that is, for any $0<\delta<\tfrac 12$, the width of the annulus on which $\varphi_\eps$ attains values between $\delta$ and $1-\delta$ is of order $\eps$.
This behavior is illustrated in Figure~\ref{fig:phase-field}.

\bigskip

\FloatBarrier

\begin{figure}[ht!]
    \centering
    \includegraphics[width=0.9\textwidth]{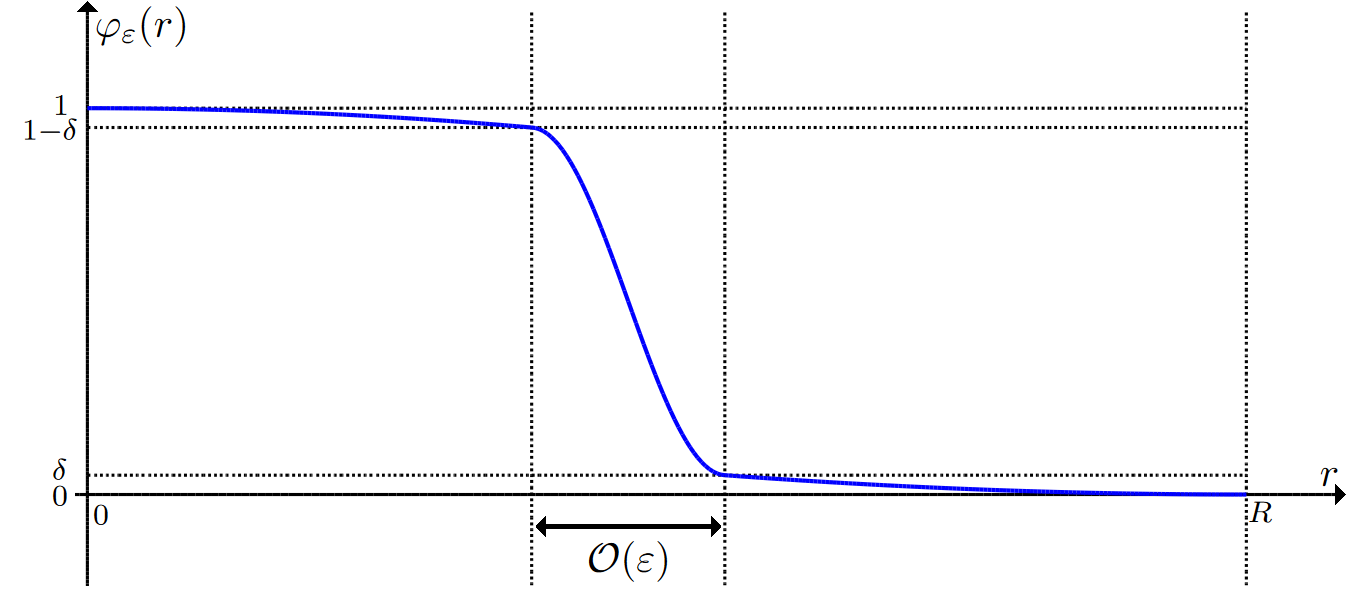}
    \caption{Schematic sketch of a minimizer $\varphi_\eps$ in radial direction $r=|x|$.}
    \label{fig:phase-field}
\end{figure}

\FloatBarrier

\bigskip

Now, we investigate the limit $\eps\to 0$. Therefore, let us fix a sequence of minimizers $(\varphi_{\eps})_{\eps>0}$ of \eqref{OP}. We intend to show that this sequence converges to the characteristic function of the ball centered at the origin with volume $m\abs{\Omega}$ and that this ball is a minimizer of a suitable limit cost functional (see Theorem~\ref{THM:EpsCon}).
To this end, we recall the most important aspects from \cite{GHKK}.  

First of all, we recall the limit eigenvalue problem, that is, the eigenvalue problem corresponding to \eqref{WeakEVe} on the sharp interface level. For details we refer again to \cite{GHKK}.
For any given $\varphi\in BV(\Omega;\left\{\pm 1\right\})$, we want to solve
\begin{subequations}
\label{EVPS}
\begin{alignat}{3}
    \label{EVPS:EQ}
    - \Lap w &= \lambda w 
    &&\qquad \text{in $E^{\varphi}$}, \\
    \label{EVPS:BC}
    w\vert_{\del E^\varphi}&= 0
    &&\qquad\text{on $\del E^{\varphi}$,}
\end{alignat}
\end{subequations}
where
\begin{align*}
    E^{\varphi}&=\left\{x\in \Omega \suchthat \varphi(x)=1\right\}.
\end{align*}
Note that, in general, $E^{\varphi}$ is only a set of finite perimeter and therefore, it merely enjoys a very weak regularity. However, the following definition turns out to be the suitable notion of weak solution as it is compatible with the sharp interface limit $\eps\to 0$ (see Proposition~\ref{PRO:ContEV}).
\begin{Def}\label{DEF:EVLimit}
    Let
    \begin{align*}
        \varphi\in \Phi_m^{0}\coloneqq \left\{ \varphi\in BV(\Omega;\left\{0,1\right\})\suchthat\textstyle\fint_{\Omega}\varphi \dLL^n=m\right\}
    \end{align*}
    be arbitrary. Then we have the following:
    \begin{enumerate}[label = \textnormal{(\alph*)}, leftmargin=*]
    \item For any given $\lambda\in\R$, a function $w\in V^{\varphi}$ is called a \textbf{weak solution} of the system \eqref{EVPS} if the weak formulation
    \begin{align}\label{LimitWF}
        \int_{\Omega}\nabla w\cdot\nabla \eta \dLL^n=\lambda\int_{\Omega}w\eta\dLL^n
    \end{align}
        is satisfied for all test functions $\eta\in V^{\varphi}$, where
    \begin{align*}
        V^{\varphi}&=\left\{ \eta \in H_0^1(\Omega) \suchthat \eta=0\text{ a.e. in }\Omega\backslash E^{\varphi}\right\}.
    \end{align*}    
    \item\label{LimitEV} A real number $\lambda^{0,\varphi}$ is called an \textbf{eigenvalue} associated with $\varphi$, if there exists at least one nontrivial weak solution $w^{0,\varphi}\in H^1_0(\Omega)$ of the eigenvalue problem \eqref{EVPS} written for $\lambda = \lambda^{0,\varphi}$.
    
    In this case, $w^{0,\varphi}$ is called an \textbf{eigenfunction} to the eigenvalue $\lambda^{0,\varphi}$.
    \end{enumerate}
\end{Def}

We recall the next proposition which is a direct consequence of \cite[Theorem 4.2]{GHKK}.
\begin{Pro}
Suppose that $\varphi\in \Phi_m^0$ with $V^{\varphi}\neq \left\{0\right\}$.
\begin{enumerate}[label = \textnormal{(\alph*)}, leftmargin=*]
\item The minimum in
\begin{align}
    \label{MINL1}
    \min \left\{\left.
    \frac{\int_{\Omega}\abs{\nabla v}^2\dLL^n}{\int_{\Omega}\abs{v}^2\dLL^n}\right\vert
    v\in V^{\varphi}\backslash\{0\}
    \right\}=:\lambda_1^{0,\varphi}.
\end{align}
is attained and any minimizer $w\in V^{\varphi}\backslash\left\{0\right\}$ is an eigenfunction of the limit problem \eqref{EVPS} to the eigenvalue $\lambda_1^{0,\varphi}$ in the sense of Definition~\ref{DEF:EVLimit}\ref{LimitEV}. 
\item $\lambda_1^{0,\varphi}>0$ is the smallest eigenvalue of the limit problem \eqref{EVPS} in the sense of Definition~\ref{DEF:EVLimit}\ref{LimitEV}.
\end{enumerate}
\end{Pro}
\begin{Rem}\label{Rem:V}
Regarding the definition of the space $V^{\varphi}$ we make note of the following:
\begin{enumerate}[label = \textnormal{(\alph*)}]
\item Note that in \cite[Theorem 4.2]{GHKK}, we are in the situation that $V^{\varphi}$ is an infinite-dimensional vector space. In the present paper, we only assume that $V^{\varphi}$ is non-trivial, but the above proposition can be established analogously using classical spectral theory.
If $V^{\varphi}=\left\{0\right\}$, we set $\lambda_1^{0,\varphi}=+\infty$, which is consistent with the above proposition.
\item\label{CapvsLebes}
We point out that the Sobolev-like space $V^\varphi$ is not new to the literature, see \cite{bucur.buttazzo.ea:minimization, de-philippis.lamboley.ea:regularity, DePhilippis} as well as \cite[Section 4.5]{henrot.pierre:shape}. 
The space $V^\varphi$ is alternatively denoted by $\tilde{H}_0^1(E^\varphi)$. More generally, for any measurable set $E\subset \Omega$, we define
\begin{align*}
    \tilde{H}_0^1(E)\coloneqq \left\{u\in H_0^1(\Omega)\suchthat u=0 \text{ almost everywhere in }\Omega\backslash E\right\}.
\end{align*}
Here, the tilde indicates that this is not the canonical generalization of classical Sobolev spaces.
In fact, if $E\subset\Omega$ is \emph{open}, the classical Sobolev space $H^1_0(E)$ can be expressed as
\begin{align}\label{ClassSobo}
    H_0^1(E)=\left\{u\in H_0^1(\Omega)\suchthat \tilde{u}=0 \text{ quasi-everywhere in }\Omega\backslash E\right\},
\end{align}
where the left hand side is understood as closure of $C_0^\infty(E)$ with respect to the $H^1$-norm and $\tilde{u}$ denotes the unique quasi-continuous representative of $u\in H_0^1(\Omega)$, see \cite[Proposition~3.3.42]{henrot.pierre:shape}.
For this reason, the canonical extension of the classical Sobolev space $H^1_0(E)$ for arbitrary sets $E\subset\Omega$ of finite perimeter is actually defined as
\begin{align*}
    H_0^1(E):=\left\{u\in H_0^1(\Omega)\suchthat \tilde{u}=0 \text{ quasi-everywhere in }\Omega\backslash E\right\}.
\end{align*}
We point out that, in general, $H^1_0(E)$ does not coincide with $\tilde{H}_0^1(E^\varphi)$.
However, if $E\subset\Omega$ is an open set with Lipschitz boundary, it actually holds that
$H_0^1(E)=\tilde{H}_0^1(E)$ (see, e.g., \cite[Remark~2.3]{de-philippis.lamboley.ea:regularity}).
However, there are two main reasons why for our optimization problem on the sharp interface level, the space $V^\varphi = \tilde{H}_0^1(E^\varphi)$ is the adequate choice.

On the one hand, we will see in Proposition~\ref{PRO:ContEV} that if a sequence of phase-fields $(\varphi_\eps)_{\eps>0}\subset H^1(\Omega)$ converges in $L^1(\Omega)$ to some $\varphi\in BV(\Omega;\left\{0,1\right\})$ then there exists a function $u\in H^1(\Omega)$ such that, along a non-relabeled subsequence, it holds that
\begin{align*}
     \underset{\varepsilon\searrow 0}{\lim}
    \norm{w^{\eps,\varphi_{\varepsilon}}-u}_{H^1(\Omega)}=0
\end{align*}
and
\begin{align*}
    \underset{\varepsilon\searrow 0}{\lim}
    \int_{\Omega}b^{\varepsilon}(\varphi_{\varepsilon})\abs{w^{\eps,\varphi_{\varepsilon}}}^2\dLL^n=
    \int_{\Omega}b^0(\varphi)\abs{u}^2\dLL^n=0.
\end{align*}
In this case, $u$ plays the role of an eigenfunction for the sharp interface problem. Recalling the construction of the coefficient function $b^\eps$ in \ref{ASS:A:2} and \ref{ASS:A:3} we have $b^0(1)=0$ and $b^0(0)=+\infty$. Thus, the condition
\begin{align*}
    \int_{\Omega}b^0(\varphi)\abs{u}^2\dLL^n&=0
\end{align*}
is equivalent to $u=0$ \emph{almost everywhere} in $\left\{\varphi=0\right\}$. This motivates the usage of the Lebesgue measure instead of the capacity.

On the other hand, one could be tempted to employ the fact that for any measurable set $E\subset\Omega$, there exists a unique quasi-open set $\omega\subset\Omega$ such that
\begin{align*}
    \tilde{H}_0^1(E)=H_0^1(\omega),
\end{align*}
see \cite[Section 4.1]{GHKK}. Even though this is a crucial relation also used in the $\Gamma$-convergence proof in \cite[Theorem 4.10]{GHKK}, it does not allow us to replace the space $V^\varphi$ with the associated space $H_0^1(\omega^\varphi)$. This is due to the fact that the limit cost functional $J_\gamma^0$ that will be defined in \eqref{DEF:J0} involves a perimeter term and it is unclear how the perimeter changes when $E^\varphi$ is replaced by the quasi-open set $\omega$, see also the discussion in \cite[Remark~2.1]{DePhilippis}.
\blk
\end{enumerate}
\end{Rem}

The following continuity result for the principal eigenvalues in the limit $\eps\to 0$ was established in \cite[Lemma 4.4]{GHKK}. 

\begin{Pro}\label{PRO:ContEV}
Let $\left(\varphi_{\eps}\right)_{\eps>0}\subset L^1(\Omega)$ with $\varphi_{\eps}\in [0,1]$ almost everywhere in $\Omega$ and suppose that $\varphi\in BV(\Omega,\left\{0,1\right\})$ with $V^{\varphi}\neq \left\{0\right\}$ such that
\begin{align*}
    \limse \norm{\varphi_{\eps}-\varphi}_{L^1(\Omega)}=0\tl{.}
\end{align*}
Moreover, we demand the additional convergence rate
\begin{align*}
   \norm{\varphi_{\eps}-\varphi}_{L^1\big(E^{\varphi}\cap 
    \left\{\varphi_{\eps}<\frac12 \right\}\big)}=\mathcal{O}(\eps).
\end{align*}
Then, there exists an eigenfunction $u\in V^{\varphi}$ of the limit problem \eqref{EVPS} to the eigenvalue $\lambda_1^{0,\varphi}$ such that
\begin{align*}
    \limse \int_{\Omega}b^{\eps}(\varphi_{\eps})\bigabs{w_1^{\eps,\varphi_{\eps}}}^2\dLL^n=
    \int_{\Omega}b^0(\varphi)\abs{u}^2\dLL^n=0,
\end{align*}
as well as
\begin{align*}
    \limse \norm{w_1^{\eps,\varphi_{\eps}}-u}_{H^1(\Omega)}=0\quad\text{and}\quad
    \limse \lambda_1^{\eps,\varphi_{\eps}}=\lambda_1^{0,\varphi}. 
\end{align*}
\end{Pro}

We point out that in \cite{GHKK}, the above result was established for any eigenvalue in the case where $V^{\varphi}$ is an infinite-dimensional vector space. However, as we only consider the principal eigenvalue, the Rayleigh quotient merely needs to be minimized over the set $V^\varphi\backslash\{0\}$ (cf.~\eqref{MINL1}). It is thus clear that the proof of \cite{GHKK} also works under the weaker assumption $V^\varphi \neq \{0\}$.

Since $\lambda_1^{0,\varphi}=+\infty$ if $V^{\varphi}=\{0\}$, the following corollary is a trivial consequence, but we state it here for the sake of completeness as this provides us with the upper semicontinuity of the principal eigenvalue even if the shape prescribed by $\varphi\in BV(\Omega,\left\{0,1\right\})$ does not admit an eigenvalue.
\begin{Cor}\label{UpperContEV}
    Let the assumptions of the previous theorem be fulfilled, but allow for the case $V^{\varphi}=\left\{0\right\}$.
    Then, it still holds that
    \begin{align*}
        \limsupse \lambda_1^{\eps,\varphi_{\eps}}\le \lambda_1^{0,\varphi}.
    \end{align*}
\end{Cor}

Finally, we consider the limit cost functional
\begin{align}\label{DEF:J0}
    &J^0_{\gamma}(\varphi)\coloneqq
    \begin{cases}
        \begin{aligned}
            &\lambda_1^{0,\varphi}+
            \gamma c_0  \left(
                P_{\Omega}(E^{\varphi})+\int_{\partial\Omega}\varphi_{\vert{\del\Omega}}\text{\,d}\mathcal{H}^{n-1}
                \right)
            &&\quad\text{if }
            \varphi\in \Phi_{m}^0,\\
            &+\infty
            &&\quad\text{if }
            \varphi\in L^1(\Omega)\backslash \Phi_{m}^0,
        \end{aligned}
    \end{cases}
\end{align}
where
$c_0=\int_0^1\sqrt{2\psi(t)}\text{\,d}t$
and $\varphi_{\vert\partial\Omega}\in L^1(\partial\Omega)$ denotes the trace of the $BV$ function $\varphi$ (see Section~\ref{Sub:BV}).

\begin{Rem}
    We note that in \cite{Owen}, where the phase-fields are subject to a more complex inhomogeneous space-dependent Dirichlet boundary condition, the corresponding term in the limit cost functional resulting from the Ginzburg--Laundau energy is written (transferred to our notation) as
    \begin{align*}
        \int_{\Omega}\abs{\nabla \Psi(\varphi)}\dLL^n+\int_{\partial\Omega}\abs{\Psi(\varphi_{\vert\del\Omega})}\dHH^{n-1}.
    \end{align*}
    Here, the function $\Psi$ is defined by
    \begin{align}\label{DEF:Root}
    \Psi(s)\coloneqq \int_0^s\sqrt{2\psi(t)}\text{\,d}t,
    \end{align}
    and $\int_{\Omega}\abs{\nabla \Psi(\varphi)}\dLL^n$ denotes the variation of $\Psi(\varphi)\in L^1(\Omega)$ as given in Section~\ref{Sub:BV}. As obviously $\Psi(\varphi)=\Psi(1)\chi_{E^{\varphi}}$ in $BV(\Omega,\{0,1\})$ we obtain
    \begin{align*}
        \int_{\Omega}\abs{\nabla \Psi(\varphi)}\dLL^n=c_0P_{\Omega}(E^{\varphi}).
    \end{align*}
    Furthermore, due to the definition of the trace in \eqref{DEF:TR} and $\varphi\in BV(\Omega,\left\{0,1\right\})$, we see also that $\varphi_{\vert\del\Omega}$ only attains the values $0$ and $1$. Hence, we have $\Psi(\varphi_{\vert\del\Omega})=\Psi(1)\varphi_{\vert\del\Omega}$ in $L^1(\partial\Omega)$, which yields
    \begin{align*}
        \int_{\partial\Omega}\abs{\Psi(\varphi_{\vert\del\Omega})}\dHH^{n-1}=c_0\int_{\partial\Omega}\varphi_{\vert\del\Omega}\dHH^{n-1}.
    \end{align*}
    Note that, as we are imposing a homogeneous Dirichlet boundary condition, we do not need to rely on the very technical construction of a recovery sequence presented in \cite{Owen} as we can simply perform a cut-off procedure as in \cite{BourdinChambolle}. 
    The idea in \cite{BourdinChambolle} is to approximate any finite perimeter set by truncated sets that are compactly contained within $\Omega$. For these truncated sets, we then perform a diffuse interface approximation in the spirit of \cite{Modica,Sternberg}.
    Using this approach, the need for the additional boundary integral in the limit cost functional can be clearly seen:
    In the course of this approximation, the boundaries of the truncated sets are getting closer and closer to the boundary of $\Omega$. Therefore, the whole boundary of the limit set has to be perceived by the limit energy. For more details, we refer to the proof of Theorem~\ref{GammaOwen} given in Section~\ref{SECT:PROOF}.
\end{Rem}

The previous discussion allows us to state the desired theorem which states the convergence of minimizers as $\eps$ tends to zero.

\begin{Thm}\label{THM:EpsCon}
Let $\varphi_0\in \Phi^0_m$ be the characteristic function of the ball centered at the origin with volume $m\abs{\Omega}$ and let $\left(\varphi_\eps\right)_{\eps>0}$ be a sequence of minimizers of \eqref{OP}.\\ Then
\begin{align*}
    \limse\norm{\varphi_\eps-\varphi_0}_{L^1(\Omega)}=0,\quad \limse J^{\eps}_{\gamma}(\varphi_\eps)=J^0_{\gamma}(\varphi_0),
\end{align*}
and $\varphi_0$ is a minimizer of $J^0_\gamma$.
\end{Thm}

The proof of Theorem~\ref{THM:EpsCon} can be found in Section~\ref{SECT:PROOF}.
As a direct consequence we finally obtain the classical Faber--Krahn theorem in our framework by sending the surface tension parameter $\gamma$ to zero.

\begin{Cor}[Faber--Krahn theorem for $BV$ functions and sets of finite perimeter]\label{COR:FK}
    Let $\varphi_0\in \Phi^0_m$ be the characteristic function of the ball centered at the origin with volume $m\abs{\Omega}$. Then, it holds that 
    \begin{align*}
        \lambda_1^{0,\varphi_0}
        =\min\left\{\lambda_1^{0,\varphi}\suchthat \varphi\in \Phi^0_m\right\}.
    \end{align*}
    Formulated
    in the framework of sets of finite perimeter it thus holds that
    \begin{align*}
        \lambda_1^0(B)=\min\left\{ \lambda_1^0(E)\suchthat E\subset\Omega \text{ measurable}, P_{\Omega}(E)<\infty, \abs{E}=m\abs{\Omega}\right\},
    \end{align*}
    where $B\subset\Omega$ is the ball centered at the origin with volume $m\abs{\Omega}$ and 
    \begin{align*} 
    \lambda_1^0(E)\coloneqq \lambda_1^{0,\chi_E}=
    \min \left\{\left.
    \frac{\int_{\Omega}\abs{\nabla v}^2\dLL^n}{\int_{\Omega}\abs{v}^2\dLL^n}\right\vert
    v\in \tilde{H}_0^1(E)\backslash\{0\}
    \right\}
    \end{align*}
    for every finite perimeter set $E\subset\Omega$.
\end{Cor}
We
point out that this result slightly extends the classical Faber--Krahn theorem, which merely states that any open ball is a minimizer among all \textit{open} sets of the same volume. Of course, Corollary~\ref{COR:FK} can also be obtained by taking the classical Faber--Krahn theorem for granted and then performing the regularization of finite perimeter sets as in the proof of Theorem~\ref{GammaOwen}. It is further worth mentioning that the classical Faber--Krahn theorem is stated without the constraint of a surrounding design domain which makes the analysis more delicate, see, for example, \cite{de-philippis.lamboley.ea:regularity, mazzoleni.pratelli:existence}.
We further want to mention that in \cite{BucurFreitas}, it was shown  that the Faber--Krahn theorem remains valid if the minimization problem is formulated on the class of \text{quasi-open} sets.

Nevertheless, the purpose of this paper is not to generalize the Faber--Krahn theorem but to understand the classical Faber--Krahn theorem within our phase-field approach as the sharp interface limit of the Faber--Krahn theorem on the diffuse interface level from Theorem~\ref{THM:MIN}.
 
In this proof, the key step is to show that $J^{\eps}_{\gamma}\overset{\Gamma}{\to} J^0_{\gamma}$ as $\eps\to 0$, that is, $J^{\eps}_{\gamma}$ converges to $J^0_{\gamma}$ in the sense of $\Gamma$-convergence. Our strategy will be similar to that in \cite{GHKK}. 

The first step in the proof of Theorem~\ref{THM:EpsCon} is to establish the $\Gamma$-convergence for slightly modified functionals $F_{\eps}^{\gamma}$, where the corresponding set of admissible phase-fields does not contain a volume constraint. In the proof, we need to revisit the construction of the recovery sequence in \cite{Sternberg} as we allow for more general potentials $\psi$. In order to tackle the Dirichlet boundary constraint hidden in $H_0^1(\Omega)$, we apply the idea of \cite[Theorem 3.1]{BourdinChambolle}. More precisely, to construct a recovery sequence for any given $\varphi\in BV(\Omega;\left\{0,1\right\})$, we approximate the corresponding set $\{\varphi=1\}$ by truncated sets which are compactly contained in $\Omega$. The $\Gamma$-convergence result is stated by the following theorem.

\begin{Thm}\label{GammaOwen}
    For any $\eps,\gamma>0$, let the functions $F_{\eps}^{\gamma}, F_{0}^{\gamma}: L^1(\Omega)\to \mathbb{R}\cup \left\{+\infty\right\}$ be defined as
    \begin{align*}
        F_{\eps}^{\gamma}(\varphi)=
        \begin{cases}
            \lambda_1^{\eps,\varphi_{\eps}} + \gamma \int_{\Omega}\frac{\eps}{2}\abs{\nabla \varphi}^2+\frac{1}{\eps}\psi(\varphi)\dLL^n
            &\quad \textup{if }\varphi\in H_0^1(\Omega;[0,1]),\\
         +\infty &\quad\textup{else},   
        \end{cases}
    \end{align*}
    and
    \begin{align*}
        F_{0}^{\gamma}(\varphi)=
        \begin{cases}
        \lambda_1^{0,\varphi}+c_0\gamma \left(P_{\Omega}(E^{\varphi})+\int_{\partial\Omega}\varphi_{\vert\del\Omega} \dHH^{n-1}\right)
            &\quad \textup{if }\varphi\in BV(\Omega;\left\{0,1\right\}),\\
         +\infty &\quad\textup{else}   .
        \end{cases}
    \end{align*}
    Then $F_{\eps}^{\gamma}\overset{\Gamma}{\to}F_{0}^{\gamma}$.
\end{Thm}

The second step is to modify the recovery sequence obtained by Theorem~\ref{GammaOwen}, as in \cite[Theorem 4.8]{GHKK}, via suitable $C^1$-diffeomorphisms such that the modified sequence is actually a recovery sequence for $J^{\eps}_{\gamma}$ satisfying the volume constraint included in $\Phi_m$. This is done in the following theorem.

\begin{Thm}\label{MainGamma}
    For any $\eps,\gamma>0$, let the functions $J^{\eps}_{\gamma}, J^0_{\gamma}: L^1(\Omega)\to \mathbb{R}\cup \left\{+\infty\right\}$ be defined as
    \begin{align*}
        J^{\eps}_{\gamma}(\varphi)=
        \begin{cases}
            \lambda_1^{\eps,\varphi}+\gamma\int_{\Omega}\frac{\eps}{2}\abs{\nabla \varphi}^2+\frac{1}{\eps}\psi(\varphi)\dLL^n
            &\quad \textup{if }\varphi\in \Phi_m,\\
         +\infty &\quad\textup{else},   
        \end{cases}
    \end{align*}
    and
    \begin{align*}
        J^0_{\gamma}(\varphi)=
        \begin{cases}
        \lambda_1^{0,\varphi}+c_0\gamma\left(P_{\Omega}(E^{\varphi})+\int_{\partial\Omega}\varphi_{\vert\del\Omega}\dHH^{n-1}\right)
            &\quad \textup{if }\varphi\in \Phi_m^0,\\
         +\infty &\quad\textup{else}   .
        \end{cases}
    \end{align*}    
    Then $J^{\eps}_{\gamma}\overset{\Gamma}{\to}J^0_{\gamma}$.
\end{Thm}

The proofs of Theorems~\ref{GammaOwen} and \ref{MainGamma} are presented in Section~\ref{SECT:PROOF}.

\begin{Rem}
    Although for our purposes we have fixed $\Omega=B_R(0)$ the results presented in Theorems~\ref{GammaOwen} and \ref{MainGamma} hold true for any bounded, open set $\Omega\subset\R^n$ with Lipschitz boundary. 
\end{Rem}




\section{Proofs}\label{SECT:PROOF}

We first assure that the basic properties of radially symmetric-decreasing rearrangements in $\R^n$ carry over to our local case.

\begin{proof}[Proof of Lemma \ref{LEM:SDR}]
    In view of Remark~\ref{REM:SDR}\ref{REM:SDR:B}, the statements \ref{item:f* measurable}--\ref{item:f* phi}, \ref{item:f* hardy littlewood} and \ref{item:f* nonexpansive} are direct consequences of the results in \cite[Sections~3.3--3.5]{lieb-loss}.
    
    To prove \ref{item:f* psi}, we use the decomposition $\Psi' = \Psi'_+ - \Psi'_-$, where $\Psi'_+ := \max(\Psi',0)$ and $\Psi'_- := -\min(\Psi',0)$ denote the positive part and the negative part of $\Psi'$, respectively.
    Now, we define
    \begin{align*}
        \Psi_1(t) := \int_0^t \Psi_+'(s) \,\mathrm ds
        \quad\text{and}\quad 
        \Psi_2(t) := \int_0^t \Psi_-'(s) \,\mathrm ds
        \quad\text{for all $t\in[0,1]$.}
    \end{align*}%
    Recalling $\Psi(0)=0$, we apply the fundamental theorem of calculus to derive the decomposition ${\Psi = \Psi_1 - \Psi_2}$.
    As the functions $\Psi_1$ and $\Psi_2$ are non-decreasing, \eqref{EQ:PSI} follows directly from \cite[Section~3.3(iv)]{lieb-loss}.
    
    To prove \ref{item:f* polya szego}, let $f\in H^1_0(\Omega;\R_0^+)$ be any function, and let $f_0:\R^n\to\R_0^+$ denote its trivial extension as in Remark~\ref{REM:SDR}\ref{REM:SDR:B}. This means that $f_0\in H^1(\R^n)$ is a non-negative function with compact support. 
    We further define
        \begin{align*}
        A:[0,\infty)\to [0,\infty), \,
        x \mapsto x^2.
    \end{align*}
    Hence $A\in C^2([0,\infty)$ is strictly increasing, $A(0)=0$ and $A^{\frac12}$ is convex.
    Thus, as all conditions are fulfilled, we can apply the first part of \cite[Theorem~1.1]{Brothers-Ziemer} and obtain
    \begin{align}\label{PSI:0}
        \int_{\R^n}\bigabs{\nabla f_0^{\ast}}^2\dLL^n\le 
        \int_{\R^n}\bigabs{\nabla f_0}^2\dLL^n,
    \end{align}%
    which directly implies \eqref{PSI} since $\Grad f_0 = 0$ and $\Grad f^*_0 = 0$ almost everywhere on $\R^n\backslash\Omega$. 
    We note that this classical Pólya--Szeg\H{o} inequality is well-known (see, e.g., \cite[Theorem 2.1.3]{Henrot}), but the strong version contained in \cite[Theorem~1.1]{Brothers-Ziemer} given below is more delicate.
    
    In addition, let us now assume that condition \eqref{COND:GRAD} holds true and that $f>0$ almost everywhere in $\Omega$. Since $f^*_0=0$ on $\R^n\backslash\Omega$, we have $\left\{f_0^{\ast}>0\right\}\subset \Omega$
    and thus,
    \begin{align*}
        \LL^n\Big(\{x\in\R^n \,\vert\, \Grad f_0^*(x) = 0\}
        \cap \big(f_0^*\big)^{-1}\big((0,\infty)\big)\Big)
        \le \LL^n\big(\{x\in\Omega \,\vert\, \Grad f^*(x) = 0\}\big)=0.
    \end{align*}%
    Therefore, \cite[Theorem~1.1]{Brothers-Ziemer} states that equality in \eqref{PSI:0} holds if and only if $f_0$ is a translate of $f_0^*$. This directly entails that equality in \eqref{PSI} holds if and only if $f$ is a translate of $f^*$. However, since $f\in H^1_0(\Omega)$ with $f>0$ almost everywhere in $\Omega=B_R(0)$, this is possible if and only if $f=f^*$ almost everywhere in $\Omega$, which proves the claim.
\end{proof}

Now we are in the position to present the proofs of our main results.

\begin{proof}[Proof of Lemma~\ref{LEM:EGL}]
In view of the definition of $\EE^\eps$ given in \eqref{DEF:EGL}, the assertion follows directly by using the Pólya--Szeg\H{o} inequality (Lemma~\ref{LEM:SDR}\ref{item:f* polya szego}) to estimate the gradient term, and by applying Lemma~\ref{LEM:SDR}\ref{item:f* psi} to the potential term.
\end{proof}

\begin{proof}[Proof of Theorem~\ref{THM:FKDif}]
    Let $\varphi\in H_0^1(\Omega;[0,1])$ be arbitrary.
    
    First of all, we derive some general inequalities.
    So, we let $w \in H^1_0(\Omega) \backslash\{0\}$ with $w\geq 0$ almost everywhere in $\Omega$ be arbitrary.
    In view of \ref{ASS:A:2}, the coefficient $b^\eps$ is continuous, decreasing and $b^{\eps}(0)=\beta^{\eps}$; therefore, we infer that the function 
    \begin{align}
        B^\eps:[0,1]\to [0,\beta^\eps], \quad
        s\mapsto b^\eps(0) - b^\eps(s)
    \end{align}
    is continuous and increasing. Hence, according to Lemma~\ref{LEM:SDR}\ref{item:f* phi}, we have
    \begin{align*}
        \big(B^\eps(\varphi)\big)^* = B^\eps(\varphi^*)
        \quad\text{and}\quad
        \big(w^2\big)^* = \big(w^* \big)^2
    \end{align*}
    almost everywhere in $\Omega$.
    Applying Lemma~\ref{LEM:SDR}\ref{item:f* level sets and p norms} and \ref{item:f* hardy littlewood}, we thus obtain
    \begin{align}
        \label{EST:B*}
        \intO b^\eps(\varphi) w^2 \dLL^n
        &= b^\eps(0) \intO w^2 \dLL^n
            - \intO B^\eps(\varphi) w^2 \dLL^n
        \notag\\
        &\ge b^\eps(0) \intO \big(w^* \big)^2 \dLL^n
            - \intO \big(B^\eps(\varphi)\big)^* \; \big(w^2\big)^* \dLL^n
        \notag\\
        &= b^\eps(0) \intO \big(w^* \big)^2 \dLL^n
            - \intO B^\eps(\varphi^*)\, \big(w^* \big)^2 \dLL^n
        \notag\\
        & = \intO b^\eps(\varphi^*) \, \big(w^* \big)^2 \dLL^n.
    \end{align}
    In particular, since $\varphi^{**} = \varphi^*$, this already entails that
    \begin{align}
        \label{EST:B**}
        \intO b^\eps(\varphi^*) w^2 \dLL^n
        \ge  \intO b^\eps(\varphi^*) \, \big(w^* \big)^2 \dLL^n.
    \end{align}
    
    These general estimates can now be used to prove the assertion $\lambda_1^{\eps,\varphi^*} \le \lambda_1^{\eps,\varphi} $.
    Indeed, we define the functional
    \begin{align}
        \label{FUNC:RAY}
        \mathcal R: H^1_0(\Omega) \backslash\{0\} \to \R,\quad
        w \mapsto 
        \frac{\intO \big(\abs{\Grad w}^2 
        + b^\eps(\varphi^*)\, w^2 \big) \dLL^n}{\norm{w}_{L^2(\Omega)}^2}\;.
    \end{align}
    
    We now consider the positive-normalized eigenfunction $w_1^{\eps,\varphi^*}$ associated with $\varphi^*$, which is obviously a minimizer of $\mathcal R$. Using \eqref{EST:B**} along with the Pólya--Szeg\H{o} inequality (Lemma~\ref{LEM:SDR}\ref{item:f* polya szego}) and Lemma~\ref{LEM:SDR}\ref{item:f* level sets and p norms}, we find that
    \begin{align*}
        \mathcal{R}\Big(\big(w_1^{\eps,\varphi^*}\big)^*\Big)
        \le
        \mathcal{R}\Big(w_1^{\eps,\varphi^*}\Big).
    \end{align*}
    Therefore, $\big(w_1^{\eps,\varphi^*}\big)^*$ is also a minimizer of $\mathcal{R}$ and thus, due to Proposition~\ref{PROP:EVP}\ref{PROP:EVP:C}, it is an eigenfunction to the eigenvalue $\lambda_1^{\eps,\varphi^*}$. As $\big(w_1^{\eps,\varphi^*}\big)^*$ is non-negative and $L^2$-normalized, this is enough to deduce
    \begin{align*}
     \big(w_1^{\eps,\varphi^*}\big)   =\big(w_1^{\eps,\varphi^*}\big)^*\quad\text{a.e. in }\Omega,
    \end{align*}
    as the eigenspace to $\lambda_1^{\eps,\varphi^*}$ is one-dimensional.
    On the other hand, the Courant--Fischer characterization (see \eqref{CHAR:CF}) yields that for any $w\in H_0^1(\Omega)\backslash \left\{0\right\}$ with $w\ge 0$ almost everywhere in~$\Omega$, we have
    \begin{align}
        \lambda_1^{\eps,\varphi^*}\le
        \frac{\intO \big(\abs{\Grad w^*}^2 
        + b^\eps(\varphi^*)\, (w^*)^2 \big) \dLL^n}{\norm{w^*}_{L^2(\Omega)}^2}\notag\\
        \le 
        \frac{\intO \big(\abs{\Grad w}^2 
        + b^\eps(\varphi)\, (w)^2 \big)\dLL^n}{\norm{w}_{L^2(\Omega)}^2}.\label{EST:STAREV}       
    \end{align}
    Here, we applied the Pólya--Szeg\H{o} inequality (Lemma~\ref{LEM:SDR}\ref{item:f* polya szego}), estimate \eqref{EST:B*} and Lemma~\ref{LEM:SDR}\ref{item:f* level sets and p norms}. Hence, choosing $w=w_1^{\eps,\varphi}$
    we use Proposition~\ref{PROP:EVP}\ref{PROP:EVP:C} to conclude that
    \begin{align}\label{EST:LAM*}
        \lambda_1^{\eps,\varphi^*}\le \lambda_1^{\eps,\varphi}
    \end{align}
    and thus, the proof is complete.
\end{proof}
\begin{proof}[Proof of Theorem~\ref{THM:MIN}]
Let $\varphi\in \Phi_m$ be any minimizer of the optimization problem \eqref{OP}.
    Combining the estimates \eqref{EST:E*} and \eqref{EST:LAM*}, we deduce that $J^\eps_\gamma(\varphi^*) \le J^\eps_\gamma(\varphi)$.
    Since $\varphi$ is a minimizer, this implies that
    \begin{align}
    \label{EQ:J}
        J^\eps_\gamma(\varphi^*)
        = J^\eps_\gamma(\varphi).
    \end{align}
    This proves that the symmetric-decreasing rearrangement $\varphi^*$ is also a minimizer of the optimization problem \eqref{OP}.
    
    Therefore, it remains to prove that the eigenfunction $w_1^{\eps,\varphi}$ and the minimizer $\varphi$ are symmetric-decreasing, meaning that $\varphi = \varphi^*$ and $w_1^{\eps,\varphi}=(w_1^{\eps,\varphi})^*$ almost everywhere in $\Omega$.
    
    First of all, using \eqref{EST:E*} and \eqref{EQ:J}, we obtain the estimate
    \begin{align*}
        \lambda_1^{\eps,\varphi}
        = J^\eps_\gamma(\varphi) - \gamma \EE^\eps(\varphi)
        \le J^\eps_\gamma(\varphi^*) - \gamma \EE^\eps(\varphi^*)
        = \lambda_1^{\eps,\varphi^*}.
    \end{align*}
    Hence, in combination with \eqref{EST:LAM*}, we conclude that
    \begin{align}\label{EQ:STAR}
        \lambda :=\lambda_1^{\eps,\varphi} = \lambda_1^{\eps,\varphi^*}.
    \end{align}
    As above let $w=w_1^{\eps,\varphi}$ be the positive-normalized eigenfunction corresponding to the principal eigenvalue $\lambda$ associated with the minimizer $\varphi$. Combining \eqref{EST:STAREV} with \eqref{EQ:STAR} we arrive at
   \begin{align}\label{EQ:EF}
   \lambda_1^{\eps,\varphi^*}=
        \frac{\intO \big( \abs{\Grad w^*}^2 
        + b^\eps(\varphi^*)\, (w^*)^2 \big)\dLL^n}{\norm{w^*}_{L^2(\Omega)}^2}
        =
        \frac{\intO \big(\abs{\Grad w}^2 
        + b^\eps(\varphi)\, (w)^2 \big)\dLL^n}{\norm{w}_{L^2(\Omega)}^2}
        =\lambda_1^{\eps,\varphi}.        
    \end{align}
    Since, according to Proposition~\ref{PROP:EVP}\ref{PROP:EVP:C}, the eigenspace associated to the eigenvalue  $\lambda_1^{\eps,\varphi^*}$ is one-dimensional, we conclude that
    \begin{align}\label{EV:COM}
        w^*=\big(w_1^{\eps,\varphi}\big)^*=w_1^{\eps,\varphi^*} \quad\text{a.e. in }\Omega.
    \end{align}
    Moreover, Proposition~\ref{PROP:EVP}\ref{PROP:EVP:C} further yields $w^*>0$ almost everywhere in~$\Omega$. 
    As $w^*$ is a symmetric-decreasing rearrangement, it follows from Lemma~\ref{LEM:SDR}\ref{item:f* measurable} that $w^*>0$ actually holds \textit{everywhere} in $\Omega$, which will be crucial in what follows.
    Since $\norm{w^*}_{L^2(\Omega)} = \norm{w}_{L^2(\Omega)} = 1$, \eqref{EQ:EF} entails that
    \begin{align}\label{EQ:BPhi}
        \intO \big(\abs{\Grad w}^2 + b^\eps(\varphi)w^2\big) \dLL^n
        = \intO \big(\abs{\Grad w^*}^2 + b^\eps(\varphi^*) (w^*)^2\big)\dLL^n.
    \end{align}
    Invoking \eqref{EST:B*}, we thus obtain
    \begin{align}
        \intO \abs{\Grad w}^2 \dLL^n 
        - \intO \abs{\Grad w^*}^2 \dLL^n
        =  \intO b^\eps(\varphi^*)(w^*)^2 \dLL^n
        - \intO b^\eps(\varphi)w^2 \dLL^n
        \le 0.
    \end{align}
    Hence, we have equality in the Pólya--Szeg\H{o} inequality (Lemma~\ref{LEM:SDR}\ref{item:f* polya szego}): 
    \begin{align}
    \label{EQ:GRADW}
        \intO \abs{\Grad w}^2 \dLL^n 
        = \intO \abs{\Grad w^*}^2 \dLL^n.
    \end{align}
    In order to prove $w=w^*$ almost everywhere in $\Omega$, we now intend to show that $w^*$ is even \textit{strictly} symmetric-decreasing.
    Therefore, we argue by contradiction and assume that this is not the case. This means that there exists a direction $x\in\R^n$ with $\abs{x}=1$ as well as $0\le s<t\le R$ such that $w^*(sx) = w^*(tx) =:c$. As the function $w^*$ is positive in $\Omega$, we deduce that $c>0$. Moreover, since $w^*$ is non-increasing in radial direction, we infer that $w^*(\tau x) = c$ for all $\tau\in[s,t]$. Because of spherical symmetry, this already implies
    \begin{align}
    \label{MAXP:1}
        w^* = c \quad\text{in}\; 
        \big\{x\in\Omega \,\big\vert\, s\le \abs{x} \le t \big\}.
    \end{align}
    We further know from Corollary~\ref{COR:REG} that $w^*\in H^2(\Omega)$ is a strong solution of the eigenvalue problem \eqref{EVP}. Hence, recalling \eqref{EQ:STAR}, we have
    \begin{align*}
        0 = \Lap w^*  = \big( b^\eps(\varphi^*) - \lambda \big) w^*  \quad\text{a.e.~in}\; 
        \big\{x\in\Omega \,\big\vert\, s< \abs{x} < t \big\}.
    \end{align*}
    As $w^*>0$ in $\Omega$ and since $b^\eps(\varphi^*) - \lambda$ is non-decreasing and defined \textit{everywhere} in $\Omega$, we infer that
    \begin{align*}
        b^\eps(\varphi^*) - \lambda &= 0 
        \quad\text{in}\; 
        \big\{x\in\Omega \,\big\vert\, s< \abs{x} < t \big\},
        \\
        b^\eps(\varphi^*) - \lambda &\ge 0
        \quad\text{in}\;\; 
        A\coloneqq\big\{x\in\Omega \,\big\vert\, s< \abs{x} < R \big\},
    \end{align*}
    which in turn implies
    \begin{align}
    \label{MAXP:2}
        \Lap w^*
        = \big( b^\eps(\varphi^*) - \lambda\big) w^* \ge 0
        \quad\text{a.e.~in}\; 
        A.
    \end{align}
    Due to \eqref{MAXP:1}, we have
    \begin{align*}
        \underset{A}{\sup}\,w^*=c,
    \end{align*}
    since $w^*$ is symmetric-decreasing.
    Applying the strong maximum principle for the Laplace operator (see, e.g., \cite[Theorem~8.19]{Gilbarg-Trudinger} with $L=\Lap$), we infer that
    \begin{align*}
        w^* = c \quad\text{in}\; A.
    \end{align*}
    However, since $c>0$, this is a contradiction to the zero-trace condition hidden in $w^*\in H^1_0(\Omega)$. We have thus proven that $w^*$ is \textit{strictly} symmetric-decreasing. 
    
    As a consequence of this strict monotonicity, we have $\Grad w^* \neq 0$ almost everywhere in $\Omega$, meaning that 
    \begin{align*}
        \LL^n\big(\big\{x\in\Omega \,\big\vert\, \Grad w^* = 0 \big\}\big) = 0.
    \end{align*}
    Recalling that $w^* >0$ in $\Omega$, we use Lemma~\ref{LEM:SDR}\ref{item:f* polya szego} along with \eqref{EQ:GRADW} to conclude  
    \begin{align*}
       w=w^*\quad\text{a.e. in }\Omega,
    \end{align*}
    meaning that $w$ is symmetric-decreasing.
    Plugging this into \eqref{EQ:BPhi} we arrive at
    \begin{align*}
       \intO (b^{\eps}(\varphi^*)-b^{\eps}(\varphi))\big(w^*\big)^2\dLL^n=0. 
    \end{align*}
    Since $w^*>0$ in $\Omega$, we infer
    \begin{align*}
        b^\eps(\varphi^*)=b^\eps(\varphi) \quad\text{a.e. in }\Omega.
    \end{align*}
    This directly implies $\varphi=\varphi^*$ almost everywhere in $\Omega$ as $b^{\eps}$ is strictly decreasing (and thus injective). This proves that $\varphi$ is symmetric-decreasing and hence, the proof is complete.
\end{proof}

\begin{proof}[Proof of Theorem~\ref{GammaOwen}]
The $\liminf$ inequality of the Ginzburg--Landau part directly carries over from \cite[Lemma 1]{Owen} as in our case, we only need to consider trivial extensions of $\varphi_{\eps},\varphi$ instead of the more complicated boundary value function $h^{\eps}$ discussed there. 
We point out that \ref{ASS:A:1} is sufficient for the proof to work, as only the continuity in $[0,1]$ and the non-negativity of the potential $\psi$ is needed to ensure that the function $\Psi$ defined in \eqref{DEF:Root} is well-defined and differentiable. Note that we actually need to include the factor $\sqrt{2}$ in the definition of $\Psi$ in order for the Modica--Mortola trick
\begin{align*}
    \int_{\Omega}\bigabs{\nabla\Psi(\varphi_{\eps})}\dLL^n=
    \int_{\Omega}\sqrt{2\psi(\varphi_{\eps})}\bigabs{\nabla \varphi_{\eps}}\dLL^n\le 
    \int_{\Omega}\Big(\frac{\eps}{2}\bigabs{\nabla \varphi_{\eps}}^2+\frac{1}{\eps}\psi(\varphi_{\eps})\Big)\dLL^n=
    E^\eps(\varphi_{\eps})
\end{align*}
in the proof of the $\liminf$ inequality to work (see, e.g., \cite[Formula (3.61)]{BloweyElliott} also).
In \cite{Owen}, the factor $2$ is used which is due to the fact that there the gradient term in the energy is not scaled by $\frac{1}{2}$.

To verify the $\liminf$ inequality for the eigenvalue term, we proceed as in \cite[Theorem 4.10]{GHKK}. However, we need to be careful with the constraints for the limit cost functional. In \cite{GHKK}, the additional constraint $\varphi\in\mathcal{U}$ was imposed which fixes a non-trivial open set $S_1$ such that $S_1\subset \left\{\varphi=1\right\}$. This guarantees that all the eigenvalues are finite. In our framework, we now additionally need to consider the case $\lambda_1^{\varphi}=+\infty$. Therefore we consider $\varphi_{\eps}\to \varphi $ in $L^1(\Omega)$ such that
\begin{align*}
    \liminfse F_{\eps}^{\gamma}(\varphi_{\eps})<+\infty.
\end{align*}
Applying  Fatou's lemma to the potential term as in \cite[Proposition 1]{Modica} (which only requires the continuity of $\psi$ demanded in \ref{ASS:A:1}), we obtain that $\varphi\in BV(\Omega;\left\{0,1\right\})$ and, up to subsequence extraction, that the sequence of eigenvalues $(\lambda^{\eps,\varphi_\eps})_{\eps>0}$ is bounded.
Hence, as in the proof of \cite[Theorem 4.10]{GHKK}, the sequence of minimizers $(v_{\eps})_{\eps>0}$ of
the problem
\begin{align*}
        \min
        \left\{
            \left.
                    \int_{\Omega}
                        \abs{\nabla v}^2
                    \dLL^n
                  +\int_{\Omega}
                        b_{\varepsilon}(\ie{\varphi})\abs{v}^2
                     \dLL^n
          \right|
            \begin{aligned}
                &v\in H^1_0(\Omega), \\
                &\norm{v}_{L^2(\Omega)}=1
            \end{aligned}
        \right\} 
\end{align*}
fulfills
\begin{alignat*}{3}
    \begin{aligned}
        \ie{v}&\rightharpoonup \overline{v}\quad\text{in }H^1_0(\Omega),
        \qquad
        \ie{v}\to \overline{v}\quad\text{in }L^2(\Omega),
        \qquad
        \ie{v}\to \overline{v}\quad\text{a.e. in }\Omega,
    \end{aligned}
\end{alignat*}
after another subsequence extraction.
Now due to the boundedness of the sequence $\int_{\Omega}b_{\varepsilon}(\ie{\varphi})\abs{\ie{v}}^2 \dLL^n$ we proceed as in \cite{GHKK} and use Fatou's lemma to infer $\overline{v}\in V^{\varphi}$. Consequently, $V^{\varphi}$ is non-trivial and $\lambda_1^{0,\varphi}<\infty$. This means that the case $\lambda_1^{0,\varphi}=+\infty$ can only occur if also 
\begin{align*}
    \liminfse F_{\eps}^{\gamma}(\varphi_{\eps})=+\infty.
\end{align*} 
Therefore, the $\liminf$ inequality is established as in \cite[Theorem 4.10]{GHKK}.

It remains to prove the $\limsup$ inequality.
First of all, as already mentioned in Remark~\ref{Rem:Pot}, we want to show that the proof given in \cite[Theorem 1]{Sternberg} for the smooth double-well potential carries over to general potentials satisfying assumptions \ref{ASS:A:1} and \ref{ASS:A:1b}.
The key step is to consider the ordinary differential equation
\begin{alignat}{2}\label{ODE}
\left\{
\begin{aligned}
    \eta^{\prime}(t)&=\sqrt{2\psi(\eta(t))},\\
    \eta(0)&=\tfrac12.
\end{aligned}   
\right.
\end{alignat}
As the right hand side is locally Lipschitz away from $\eta(t)=0$ and $\eta(t)=1$, the Picard--Lindelöf theorem provides the existence of a unique maximal solution $\eta$ on an open interval $(t_0,t_1)$ with suitable $t_0 \in \R \cup\left\{-\infty\right\}$ and $t_1 \in \R \cup\left\{+\infty\right\}$, satisfying 
\begin{align}\label{limTmax}
    \underset{t\to t_1}{\lim}\eta(t)=1 
    \quad\text{and}\quad
    \underset{t\to t_0}{\lim}\eta(t)=0.
\end{align}
Moreover, since the right-hand side is non-negative, we know that $\eta$ is non-decreasing.
Now, depending on the choice of the potential $\psi$, the values $0$ and/or $1$ can either be reached in finite time, (i.e., $t_0$ or $t_1$ are finite), or the solution tends to those values asymptotically, (i.e., $t_0$ or $t_1$) are non-finite.
If the solution $\eta$ satisfies $\eta(t_0) = 0$ for some finite $t_0<0$ (or $\eta(t_1)=1$ for some finite $t_1>0$), it holds that $\eta(t)=0$ for all $t\le t_0$ (or $\eta(t)=1$ for all $t\ge t_1$). In particular, in any case, the solution exists for all $t\in\R$.
These properties follow from classical ODE theory, exploiting that, due to \ref{ASS:A:1}, the right-hand side of \eqref{ODE} is strictly positive whenever $\eta(t)\in (0,1)$.

As in \cite[(1.22)]{Sternberg}, we construct the profile function $\rho_{\eps}^P:\R\to [0,1]$ by defining
\begin{alignat}{2}\label{Profile}
    \rho_{\eps}^P(t)\coloneqq
    \begin{cases}
        1
        &\quad\text{for } t >2\sqrt{\eps},\\
       1+\left(1-\eta\left(\eps^{-\frac12}\right)\right)\left(\frac{t-2\sqrt{\eps}}{\sqrt{\eps}}\right),
        &\quad\text{for }\sqrt{\eps} \le t\le 2\sqrt{\eps},\\
        \eta\left(\frac{t}{\eps}\right),
        &\quad\text{for } \abs{t} \le \sqrt{\eps},\\
        \eta\left(-\eps^{-\frac12}\right)\left(\frac{t+2\sqrt{\eps}}{\sqrt{\eps}}\right),
        &\quad\text{for }-2\sqrt{\eps} \le t\le -\sqrt{\eps},\\    
        0,
        &\quad\text{for }t < -2\sqrt{\eps}.
    \end{cases}
\end{alignat}
The idea behind these profiles is to use the solution of \eqref{ODE} and possibly linearly interpolate the values where $\eta$ is close to $0$ or $1$. This interpolation is necessary to obtain a transition from $0$ to $1$ on a \textit{finite} interval scaling suitably with $\eps$, even though the solution of \eqref{ODE} does possibly not reach these values in finite time.

In case $\eta$ actually reaches the values $0$ and/or $1$, the interpolations in the second and fourth line of this definition become trivial and therefore negligible, provided that $\eps>0$ is sufficiently small.

In case the values $0$ and/or $1$ are only reached asymptotically, we still need the following exponential convergence rates in order for the proof in the spirit of \cite{Sternberg} to work out: 
\begin{itemize}
    \item If the value $1$ is only reached asymptotically, there exists $T_1>0$ as well as constants $C_1,a_1>0$ such that
    \begin{align}
    \label{EST:C1}
        \abs{1-\eta(t)}\le C_1 \exp(-a_1 t)\quad \text{for all $t\ge T_1$.}
    \end{align}
    \item If the value $0$ is only reached asymptotically, there exists $T_0>0$ as well as constants $C_0,a_0>0$ such that
    \begin{align}
    \label{EST:C0}
        \abs{\eta(t)}\le C_0 \exp(-a_0 t)\quad \text{for all $t\le -T_0$.}
    \end{align}
\end{itemize}
We will only prove estimate \eqref{EST:C1} since estimate \eqref{EST:C0} can be established completely analogously.
Therefore, we assume that $\eta$ reaches the value $1$ only asymptotically.
Due to the monotonicity of $\eta$, we have $\eta(t) \in \big[\tfrac 12,1 \big)$ and thus $\psi\big(\eta(t)\big)>0$ for all $t\in[0,\infty)$. Hence, $\eta$ is twice continuously differentiable with
\begin{align*}
    \eta^{\prime\prime}(t)=
    \frac{1}{2\sqrt{2\psi(\eta(t))}}2\psi^{\prime}(\eta(t))\eta^{\prime}(t)=
    \psi^{\prime}(\eta(t)) \quad\text{for all }t\in[0,\infty).
\end{align*}
Combining \eqref{limTmax} and \eqref{ODE}, we further deduce
\begin{align*}
    \underset{t\to \infty}{\lim} \eta^{\prime}(t)=0.
\end{align*}
Hence, for any $\xi\in [0,\infty)$, we have
\begin{align}\label{DifEq}
    \int_{\xi}^{\infty}\psi^{\prime}(\eta(t)) \text{\,d} t
    =-\eta^{\prime}(\xi).
\end{align}
Let us now assume that $\psi^{\prime}(1)\neq 0$. 
Since $\psi\colon [0,1]\to\R$ possesses a local minimum at $1$, this already entails $\psi^{\prime}(1)< 0$.
Hence, due to the continuity of $\psi'$, we have 
\begin{align*}
    \psi^{\prime}(s)<\frac{1}{2}\psi'(1)<0
\end{align*}
for all $s\in (0,1)$ in a suitably small neighborhood around $1$.
However, this is an obvious contradiction to the finiteness of the integral in \eqref{DifEq}. 
We thus conclude
\begin{align}\label{Psi10}
    \psi^{\prime}(1)=0.
\end{align}
This equality will now be the crucial ingredient in applying the comparison principle for ODEs.
Recalling that $\psi\in C^2([0,1])$, for any $s\in (0,1]$, we consider the Taylor expansion
\begin{align}\label{TayPsi}
    \psi(s)=\psi(1)+\psi^{\prime}(1)(s-1)
    +\frac{1}{2}\psi^{\prime\prime}(\xi_s)(s-1)^2
    =\frac{1}{2}\psi^{\prime\prime}(\xi_s)(s-1)^2,
\end{align}
for a $\xi_s$ between $1$ and $s$. 
In the light of \eqref{Psi10} and \ref{ASS:A:1b}, there exist $c,\delta>0$ such that
\begin{align*}
   \abs{\psi^{\prime\prime}(x)}\ge c >0, \quad\text{for all }x\in (1-\delta, 1+\delta)\cap[0,1].
\end{align*}
Hence, defining $a_1\coloneqq \sqrt{c}$, we obtain
\begin{align}\label{RHSEstimate}
    \sqrt{2\psi(s)}\ge a_1(1-s)
\end{align}
for all $s\in  (1-\delta, 1+\delta)\cap[0,1].$
Due to \eqref{limTmax}, there exists $t_1\in \R$ such that 
\begin{align}\label{initEta}
    s_1\coloneqq \eta(t_1)\in (1-\delta, 1+\delta)\cap(0,1)
\end{align}
We now consider the initial value problem
\begin{alignat}{2}\label{ODE2}
\left\{
\begin{aligned}
    \mu^{\prime}(t)&=a_1(1-\mu(t)),\\
    \mu(t_1)&=s_1,
\end{aligned}
\right.
\end{alignat}
which possesses the unique global solution
\begin{align*}
    \mu:\R \to \R, \quad \mu(t)=(s_1-1)\exp(-a_1(t-t_1))+1.
\end{align*} 
On the other hand, combining \eqref{ODE} and \eqref{initEta} with \eqref{RHSEstimate} and recalling that $\eta$ is non-decreasing  we infer that $\eta: \R \to \R$ satisfies
\begin{alignat}{2}\label{ODE3}
\left\{
\begin{aligned}
    \eta^{\prime}(t)&\ge a_1(1-\eta(t))\quad \text{for } t\ge t_1,\\
    \eta(t_1)&=s_1.
\end{aligned}   
\right.
\end{alignat}
Equations \eqref{ODE2} and \eqref{ODE3} directly imply
\begin{alignat*}{2}
\left\{
\begin{aligned}
    (\mu-\eta)^\prime(t)&\le - a_1(\mu-\eta)(t)\quad \text{for } t\ge t_1,\\
    (\mu-\eta)(t_1)&=0
\end{aligned}   
\right.
\end{alignat*}
and thus
\begin{align*}
    \eta(t)\ge \mu(t) \quad\text{for all } t\ge t_1,
\end{align*}
 as a direct consequence of Gronwall's lemma.

Hence, we conclude
\begin{align}\label{expDecay}
    \abs{\eta(t)-1}=1-\eta(t)\le (1-s_1)\exp(-a_1(t-t_1)) \quad\text{for all } t\ge t_1,
\end{align}
which proves \eqref{EST:C1} with $\bar C_1:=(1-s_1) \exp(a_1 t_1) >0$ and $T_1:=t_1$. 

Now, the estimates \eqref{EST:C1} and \eqref{EST:C0} allow us to continue as in \cite{Sternberg}. 
Therefore, we will need to approximate any general finite perimeter set by a suitable sequence of smooth sets. The reason for this approximation is that we know from the proof of \cite[Theorem 1]{Sternberg} that for any smooth, bounded, open set $E\subset \R^d$ having finite perimeter and satisfying the transversality condition
\begin{align*}
    \HH^{n-1}(\partial E\cap \partial\Omega)=0,
\end{align*}
there exists a recovery sequence $(\varphi_{\eps})_{\eps>0}\subset H^1(\Omega;[0,1])$ which satisfies
\begin{align}\label{LimsupModica}
\left\{
\begin{aligned}
 \limsupse 
 \int_{\Omega}
    \Big(
    \frac{\eps}{2}\abs{\nabla{\varphi}_{\eps}}^2+\frac{1}{\eps}\psi({\varphi}_{\eps})
    \Big)
 \dLL^n
 &\le c_0\text{Per}_{\Omega}(E),\\
 \norm{\varphi_{\eps}-\chi_E}_{L^1(\Omega)}&=\mathcal{O}(\eps).
\end{aligned} 
\right.
\end{align}
We point out that the above $L^1$-convergence rate can be obtained by following the line of argument in Step 1 of \cite[Theorem 1]{Sternberg}. The key fact is that, due to \eqref{EST:C1} and \eqref{EST:C0}, the profile $\rho_{\eps}^P$ converges in the interpolation parts exponentially to $0$ and $1$, respectively. In the middle part, $\rho_{\eps}^P$ is scaled with $\eps$ such that, using the coarea formula and a change of variables, we obtain the desired rate.

Obviously, this recovery sequence is not yet admissible as (in general) it does not fulfill the homogeneous Dirichlet boundary condition hidden in $H_0^1(\Omega)$. Following the idea in \cite[Theorem 3.1]{BourdinChambolle}, we make the following observation: 
If the finite perimeter set $E$ is \textit{compactly} contained in $\Omega$, then $\left(\varphi_{\eps}\right)_{\eps>0}\subset H_0^1(\Omega)$ is guaranteed provided that $\eps>0$ is sufficiently small. This directly follows from the construction of $\varphi_{\eps}$ via the optimal profile $\rho_\eps^P$ which vanishes in all points $t<-2\sqrt{\eps}$. This means that outside of a small tubular neighborhood around the boundary of $E$ we indeed have $\varphi_{\eps}=0$.
Therefore, we now approximate any finite perimeter set $E\subset\Omega$ by smooth, open finite perimeter sets which are compactly contained in $\Omega$. Although the line of argument is outlined in the proof of \cite[Theorem 3.1]{BourdinChambolle}, we highlight the key steps in order to present a comprehensive proof.

Let now $\varphi\in BV(\Omega;\left\{0,1\right\})$ be arbitrary and $E:=\{\varphi=1\}$.
In what follows,
we use the notation 
\begin{align*}
    E^{\Omega}:=E\cap\Omega, \quad
    \lambda_1^0(E)\coloneqq \lambda_1^{0,\chi_{E^{\Omega}}}\quad\text{and}\quad
    F_{0}^{\gamma}(E)\coloneqq F_{0}^{\gamma}(\chi_{E^{\Omega}}).
\end{align*}

As mentioned above, in order to construct a recovery sequence in $H_0^1(\Omega;[0,1])$,
we now approximate the finite perimeter set $E$ by a sequence $(E_k)_{k\in\N}$ of bounded, smooth, open sets $E_k\subset \R^n$ fulfilling
\begin{align}
\left\{
\begin{aligned}\label{PropEk}
    \HH^{n-1}(\partial E_k\cap \partial \Omega)&=0,\\
    \text{Per}_{\Omega}(E_k)&\to \text{Per}_{\Omega}(E)\quad \text{for }k\to\infty,\\
    \chi_{E_k^\Omega}&\to \chi_{E} \quad\text{in }  L^1(\Omega)\text{ for }k\to\infty,\\
	\underset{k\to\infty}{\limsup}\,\lambda_1^0(E_k)&\le\lambda_1^0(E).      
\end{aligned} 
\right.
\end{align}
Such a sequence is constructed in the proof of \cite[Theorem 4.10]{GHKK} relying on the ideas of \cite{BogoselOudet, Rindler, Modica}.
Note that the second and the third property of \eqref{PropEk} mean that
\begin{align}\label{chiStrict}
    \chi_{E_k^\Omega}\to\chi_{E} \text{ \textit{strictly} in } BV(\Omega)
\end{align}
as $k\to\infty$, see Subsection~\ref{Sub:BV}.
Due to the continuity of the trace operator with respect to strict $BV$-convergence, we further know
\begin{align*}
    \bignorm{\chi_{E_k^\Omega}}_{L^1(\partial\Omega)}\to \norm{\chi_{E}}_{L^1(\partial\Omega)}
    =\int_{\partial\Omega}\varphi_{\vert\del\Omega} \dHH^{n-1},
\end{align*}
for $k\to\infty$.

Now, the crucial idea of \cite[Theorem 3.1]{BourdinChambolle} is to perform a further approximation by cutting 
off $E_k$ in a tubular neighborhood around the boundary of $\Omega$ such that the truncated set is compactly contained in $\Omega$.
As we fix $k\in\N$ in what follows, we omit this index for a cleaner presentation.\\
For any $\delta>0$, the truncated set is defined as $E^{\delta}\coloneqq E\cap B_{\delta}$
with $$B_\delta\coloneqq \left\{x\in \Omega \suchthat \mathrm{dist}(x,\partial\Omega)>\delta\right\}.$$ Obviously, $E^{\delta}$ is compactly contained in $\Omega$ and it also is a set of finite perimeter.
We now intend to show that
\begin{align}\label{supSharpChambolle}
\left\{
\begin{aligned}
    &\;\underset{\delta\to 0}{\limsup}\;F_{0}^{\gamma}(E^{\delta})\le F_{0}^{\gamma}(E),\\
    &\;\underset{\delta\to 0}{\lim}\;\bignorm{\chi_{E^{\delta}}-\chi_{E}}_{L^1(\Omega)}=0.
\end{aligned} 
\right.
\end{align}
Then, applying a diagonal sequence argument will yield the desired $\limsup$ inequality, see below.

The $L^1$-convergence in \eqref{supSharpChambolle} is clear by construction.
To establish the first line of \eqref{supSharpChambolle}, we consider the eigenvalue term and the perimeter term separately.

For the eigenvalue term, we need to rely on the concept of $\gamma$-convergence (see \cite[Definition~3.3.1]{Bucur}), which was also a crucial tool in \cite[Theorem~4.10]{GHKK}. 
First of all, by using the characterization of $\gamma$-convergence via Mosco convergence (see \cite[Proposition~4.5.3]{Bucur}), we can show
\begin{align*}
    B_\delta\overset{\gamma}{\to}\Omega,
\end{align*}
for $\delta\to 0$.
Concerning the Mosco 1 condition, we note that for any $\phi\in H_0^1(\Omega)$ we find a sequence $(\phi_\delta)_{\delta>0}\subset C_0^\infty(\Omega)$  with
\begin{align*}
    \phi_\delta\to\phi\quad\text{in }H_0^1(\Omega).
\end{align*}
Due to the construction of $B_\delta$ and the fact that each $\phi_\delta$ has compact support in $\Omega$, we know $\phi_\delta\in H_0^1(B_\delta)$ (after possibly relabeling the index $\delta$).
The Mosco 2 condition is clear as $B_\delta\subset\Omega$.
Hence, from the fact that $\gamma$-convergence is stable with respect to intersection (cf. \cite[Proposition~4.5.6]{Bucur}), we infer the convergence
\begin{align}\label{EGamCon}
    E^\delta=E\cap B_\delta\overset{\gamma}{\to}E\cap \Omega=E^\Omega.
\end{align}
Note that at this stage we have to be careful not to confuse the notion of eigenvalues, as continuity with respect to $\gamma$-convergence is only formulated for the notion of eigenvalues defined on the classical Sobolev space, that is,
\begin{align*}
    \overline{\lambda}_1(\omega):=\min \left\{\left.
    \frac{\int_{\Omega}\abs{\nabla v}^2\dLL^n}{\int_{\Omega}\abs{v}^2\dLL^n}\right\vert
    v\in H_0^1(\omega)\backslash\{0\}
    \right\},
\end{align*}
with $\omega\subset \Omega$ quasi-open. For the $\gamma$-continuity
of this ``classical'' eigenvalue we refer to \cite[Corollary~6.1.8, Remark~6.1.10]{Bucur}.
Recalling the characterization in \eqref{MINL1}, our notion of the limit eigenvalue is defined as
\begin{align*}
    \lambda_1^0(\tilde{E})=\min \left\{\left.
    \frac{\int_{\Omega}\abs{\nabla v}^2\dLL^n}{\int_{\Omega}\abs{v}^2\dLL^n}\right\vert
    v\in \tilde{H}_0^1(\tilde{E})\backslash\{0\}
    \right\},
\end{align*}
for any finite perimeter set $\tilde{E}\subset \Omega$, since by construction, $V^{\chi_{\tilde{E}}}=\tilde{H}_0^1(\tilde{E})$ (see also Remark~\ref{Rem:V}\ref{CapvsLebes}). Nevertheless, as $E =E_k \subset \R^d$ and $B^\delta\subset \Omega$ are smooth open sets, we infer from the theory recalled in Remark~\ref{Rem:V} that 
\begin{align*}
    \tilde{H}_0^1(E\cap B_\delta)=H_0^1(E\cap B_\delta).
\end{align*}
Thus, in our case, $\lambda_1^0(E^\delta)=\lambda_1(E^\delta)$. In the light of \eqref{EGamCon}, we use the $\gamma$-continuity of eigenvalues (cf. \cite[Corollary 6.1.8, Remark 6.1.10]{Bucur}) to obtain
\begin{align}
    \label{LS:EIG}
    \underset{\delta\searrow 0}{\lim}\; \lambda_1^0(E^\delta)=\lambda_1^0(E).
\end{align}
For the perimeter term, we obtain
\begin{align}
    \label{EST:PER:1}
    P_{\Omega}(E^\delta)=P_{\Omega}(E\cap B_\delta)\le \HH^{n-1}(\partial(E\cap B_\delta))
\end{align}
due to \cite[Proposition 3.62]{AFP}. Applying \cite[Proposition 2.95]{AFP}, we deduce that
\begin{align}\label{boundMaes}
    \HH^{n-1}(\partial E\cap \partial B_{\delta})=\HH^{n-1}\big(\partial E\cap \big\{\mathrm{dist}(\cdot,\partial\Omega)=\delta\big\}\big)=0.
\end{align}
for almost every $\delta>0$.
Using the simple fact
\begin{align*}
    \partial(E\cap B_{\delta})\subset (\partial E\cap \overline{B_{\delta}})\cup (\partial B_{\delta}\cap \overline{E}),
\end{align*}
and exploiting \eqref{boundMaes},
we arrive at
\begin{align}
    \label{EST:H}
    \HH^{n-1}(\partial(E\cap B_\delta))
    &\le \HH^{n-1}(\partial E \cap B_{\delta})+\HH^{n-1}(E \cap \partial B_{\delta})
    \notag\\
    &= P_{B_{\delta}}(E)+\HH^{n-1}\big(E \cap \big\{\mathrm{dist}(\cdot,\partial\Omega)=\delta\big\}\big).
\end{align}
Here, for the equality in the second line, the smoothness of $E$ is crucial.
The first summand in the second line of \eqref{EST:H} side can be expressed as
$
   P_{B_{\delta}}(E)=\abs{D\chi_{E}}(B_{\delta}),
$
where $\abs{D\chi_{E}}$ denotes the total variation of the Radon-measure $D\chi_E$ associated with $\chi_E\in BV(\Omega)$ (see, e.g., \cite[Chapter~12]{Maggi}). From the $\sigma$-additivity of $\abs{D\chi_{E}}$, we directly infer
\begin{align}
    \label{EST:PER:2}
    \underset{\delta\to 0}{\lim}\;P_{B_{\delta}}(E)
    = \underset{\delta\to 0}{\lim}\;\abs{D\chi_{E}}(B_{\delta})
    = \abs{D\chi_{E}}(\Omega)
    = P_{\Omega}(E).
\end{align}
For the second summand in the second line of \eqref{EST:H}, we use the transversality condition $\HH^{n-1}(\partial E \cap \partial\Omega)=0$ in order to apply \cite[Lemma~13.9]{Rindler}. This yields
\begin{align}
    \label{EST:PER:3}
    \underset{\delta\to 0}{\lim}\;\HH^{n-1}\big(E \cap \big\{\mathrm{dist}(\cdot,\partial\Omega)=\delta\big\}\big)
    =\HH^{n-1}(E \cap \partial\Omega)
    =\int_{\partial\Omega}\varphi_{\vert\partial\Omega}\dHH^{n-1}.
\end{align}
Noticing that the trace of $\varphi_\delta$ vanishes on $\partial\Omega$ in the sense of Section~\ref{Sub:BV}, we combine \eqref{EST:PER:1}--\eqref{EST:PER:3} to obtain
\begin{align}
    \label{LS:PER}
    \underset{\delta\to 0}{\lim\sup}\; P_\Omega(E^\delta) \le P_\Omega(E).
\end{align}
Combining \eqref{LS:EIG} and \eqref{LS:PER}, we eventually conclude with \eqref{supSharpChambolle}.

Now, we close the proof by means of a final diagonal sequence argument. Therefore, we reinstate the index $k$.
Note that, without loss of generality, we can assume $E_k^{\delta}$, which is compactly contained in $\Omega$, is smooth by performing again the approximation of \eqref{PropEk}. 
We point out that by performing this approximation of $E_k^{\delta}$, the approximation sets are still compactly contained in $\Omega$. This is because the corresponding proof in \cite[Theorem 4.10]{GHKK} is based on classical convolution with mollifiers and thus, the set $E_k^{\delta}$ is only modified up to a small tubular neighborhood.

As we now take for granted that the sets $E_k^\delta$ are smooth, there exists a recovery sequence $\varphi_{\eps}^{k,\delta}\subset H_0^1(\Omega;[0,1])$ fulfilling \eqref{LimsupModica} with $E$ replaced by $E_k^{\delta}$. Using the convergence rate and the upper semicontinuity of eigenvalues provided by Corollary~\ref{UpperContEV}, we further know for any for fixed $k\in \N$ and $\delta>0$,
\begin{align*}
    \limsupse\lambda_1^{\eps,\varphi_{\eps}^{k,\delta}}\le \lambda_1(E_k^{\delta})
\end{align*}
and consequently,
\begin{align*}
    \limsupse F_{\eps}^{\gamma}(\varphi_\eps^{k,\delta})\le F_{0}^{\gamma}(E_k^{\delta}).
\end{align*}
Now, according to \eqref{PropEk} and \eqref{supSharpChambolle}, for every $k\in \N$ we can find a sufficiently small $\delta_k>0$ such that
\begin{align*}
   \underset{k\to\infty}{\limsup}\,F_{0}^{\gamma}(E_k^{\delta_{k}})&\le F_{0}^{\gamma}(E)\quad \text{and}\\
   \underset{k\to\infty}{\lim}\bignorm{\chi_{E_k^{\delta_{k}}}-\chi_E}_{L^1(\Omega)}&=0.
\end{align*}
This in turn allows us now to also choose $\eps_k>0$ small enough such that finally
\begin{align*}
    \underset{k\to\infty}{\limsup}\,F_{\eps_k}(\varphi_{\eps_k}^{k,\delta_{k}}) &\le F_{0}^{\gamma}(E)\quad\text{and}\\
    \underset{k\to\infty}{\lim}\bignorm{\varphi_{\eps_k}^{k,\delta_{k}}-\chi_E}_{L^1(\Omega)}&=0.
\end{align*}
Thus, the proof is complete.
\end{proof}

\begin{proof}[Proof of Theorem~\ref{MainGamma}]
    We have $\Phi_m\subset H_0^1(\Omega;[0,1])$ and $\Phi_m^0\subset BV(\Omega;\{0,1\})$ and we know that the volume constraint is preserved under $L^1$-convergence. Hence, the $\liminf$ inequality is a direct consequence of Theorem~\ref{GammaOwen} as now there are less admissible sequences.
    
    It remains to prove the $\limsup$ inequality---more precisely, that for every $\varphi\in \Phi_m^0$ there exists a sequence $(\tilde\varphi_{\eps})_{\eps>0}\subset \Phi_m$ fulfilling
    \begin{align}
        \limse \norm{\tilde\varphi_{\eps}-\varphi}_{L^1(\Omega)}&=0,\label{L1Phi}\\
        \limsupse J^{\eps}_{\gamma}(\tilde\varphi_{\eps})&\le J^0_{\gamma}(\varphi)\label{LimSupF}.
    \end{align}
    For any $\varphi\in \Phi_m^0$, a recovery sequence $(\varphi_{\eps})_{\eps>0}\subset H^1_0(\Omega;[0,1])$ for the functional $F_{\eps}^{\gamma}$ was
    constructed in Theorem~\ref{GammaOwen}. It was shown that this recovery sequence converges in $L^1(\Omega)$ and fulfills the $\limsup$ inequality for $F_{\eps}^{\gamma}$. 
    Now, our goal is to carefully modify this recovery sequence such that it preserves these properties but additionally fulfills the mean value constraint.
    For this modification, we proceed completely analogously to the proof of \cite[Theorem~4.8]{GHKK}. 
    Therefore, we only briefly sketch the most important steps.
    Since $\varphi\in \Phi_m^0$, it is non-constant. Hence, we can find a function $\B{\xi}\in C_0^1(\Omega,\R^n)$ such that
    \begin{align*}
        \int_{\Omega}\varphi\nabla\cdot \B{\xi}\dLL^n>0.
    \end{align*}
    For any $s\in\R$, we define the function
    \begin{align*}
        T_{s}:\R^n &\to\R^n,\\
        x&\mapsto x+s\B{\xi}(x),
    \end{align*}
    which is a $C^1$-diffeomorphism if $s$ is sufficiently small.
    By means of the implicit function theorem, we deduce that for any sufficiently small $\eps>0$, there exists $s(\eps)\in \R$ such that
    \begin{align*}
        \tilde{\varphi}_{\eps}\coloneqq \varphi_{\eps}\circ T_{s(\eps)}^{-1} \in \Phi_m,
    \end{align*}
    and $s(\eps)\to 0$ as $\eps\to 0$.
    In particular, the property $\tilde{\varphi}_{\eps}\in H_0^1(\Omega)$ holds since for $x$ close to $ \partial\Omega$, we have $T_{s(\eps)}(x)=x$ due to the fact that $\B{\xi}$ has compact support in $\Omega$.
    Eventually, we show that the sequence $(\tilde\varphi_{\eps})_{\eps>0}$ satisfies the properties \eqref{L1Phi} and \eqref{LimSupF} and thus, it is a suitable recovery sequence.
\end{proof}

\begin{proof}[Proof of Theorem~\ref{THM:DifInt}]
    In the proof of Theorem~\ref{MainGamma}, for any admissible $\varphi\in \Phi_m^0$, we have constructed a recovery sequence $(\tilde\varphi_{\eps})_{\eps>0}$. In particular, this implies that the cost functional $J_\gamma^\eps$ is bounded uniformly in $\eps$ along any sequence $(\varphi_\eps)_{\eps>0}$ of minimizers to the optimization problem \eqref{OP}. Consequently, there exists a constant $C>0$ independent of $\eps$ and $\gamma$ such that
    \begin{align}\label{EST:PsiEps}
        \int_{\Omega}\psi(\varphi_\eps)\dLL^n
        \le \eps \EE^\eps(\varphi_\eps)
        \le \frac\eps\gamma J_\gamma^\eps(\varphi_\eps)
        \le \frac{C\eps}{\gamma},
    \end{align}
    for all $\eps>0$. 
    Note that here the constant $C>0$ is universal in the sense that it is independent of the sequence of minimizers $(\varphi_\eps)_{\eps > 0}$ because the sequence $(J^\eps_{\gamma}(\varphi_\eps))_{\eps >0}$ is bounded independently of that choice.
    Recalling that $\psi\in C^2([0,1])$ is non-negative,
    we thus have
    \begin{align}\label{EST:PsiEps2}
        &\Big(\,\underset{[\delta,1-\delta]}{\min}\;\psi\,\Big) \; \LL^n\left(\left\{\delta\le \varphi_\eps\le 1-\delta\right\}\right)
        \le
            \int_{\left\{\delta\le \varphi_\eps\le 1-\delta\right\}}\psi(\varphi_\eps)\dLL^n
        \notag\\
        &\quad \le
            \int_{\Omega}\psi(\varphi_\eps)\dLL^n
        \le 
            \frac{C\eps}{\gamma},.
    \end{align}
    Since $\psi>0$ on $[\delta,1-\delta]$ due to \ref{ASS:A:1}, the assertion directly follows. 
\end{proof}

\begin{proof}[Proof of Theorem \ref{THM:EpsCon}]
We apply the compactness of the Ginzburg Landau energy of \cite[Proposition 3(a)]{Modica} which only relies on the fact that
\begin{align*}
\Psi:[0,1]\to [0,c_0], \quad
    s\mapsto \int_{0}^s\sqrt{2\psi(t)}\text{\,d}t,
\end{align*}
is invertible. This is true since due to \ref{ASS:A:1}, we have $\psi>0$ in $(0,1)$.
Consequently, there exists a function $\varphi_0\in BV(\Omega,\left\{0,1\right\})$ such that
\begin{align}
    \label{CONV:L1}
    \limse\norm{\varphi_{\eps}-\varphi_0}_{L^1(\Omega)}=0
\end{align}
along a non-relabeled subsequence of $(\varphi_{\eps})_{\eps>0}$.
We further recall that $\varphi_\eps = \varphi_\eps^{\ast}$ according to Theorem~\ref{THM:MIN}. 
Hence, the non-expansivity of the rearrangement (see Lemma \ref{LEM:SDR}\ref{item:f* nonexpansive}) yields
\begin{align*}
    \int_{\Omega}\abs{\varphi_0-\varphi_{0}^{\ast}}\dLL^n
    &\le
    \int_{\Omega}\abs{\varphi_0-\varphi_\eps}\dLL^n
        + \int_{\Omega}\abs{\varphi_\eps^{\ast}-\varphi_{0}^{\ast}}\dLL^n
    \\
    &\le
    2 \int_{\Omega}\abs{\varphi_\eps-\varphi_{0}}\dLL^n.
\end{align*}%
Hence, we infer $\varphi_0=\varphi_0^{\ast}$ almost everywhere in $\Omega$. As the mean value is preserved under $L^1$-convergence this is already enough to deduce that $\varphi_0$ is the characteristic function of the ball centered at the origin with volume $m\abs{\Omega}$.
Obviously the limit $\varphi_0$ does not depend on the choice of the subsequence of $(\varphi_{\eps})_{\eps>0}$. Hence, the convergence \eqref{CONV:L1} even holds for the \textit{whole} sequence.

Eventually, using Theorem~\ref{MainGamma} which states that $J^{\eps}\overset{\Gamma}{\to}J^0$, we conclude that $\varphi_0$ is a minimizer of $J^0_\gamma$ as $\Gamma$-convergence implies the convergence of minimizers.
\end{proof}

\section*{Acknowledgments}
    Paul Hüttl and Patrik Knopf gratefully acknowledge the support by the Graduiertenkolleg 2339 IntComSin of the Deutsche Forschungsgemeinschaft (DFG, German Research Foundation) – Project-ID 321821685.
	Tim Laux has received funding from the Deutsche Forschungsgemeinschaft (DFG, German Research Foundation) under Germany's Excellence Strategy -- EXC-2047/1 -- 390685813.


\footnotesize\setlength{\parskip}{0cm}
\renewcommand{\sc}{\scshape}
\bibliographystyle{siam}
\bibliography{HKL}

\end{document}